\documentclass[11pt, a4paper]{amsart}
\usepackage[utf8]{inputenc}
\usepackage[usenames,dvipsnames]{xcolor}
\usepackage{amsmath}
\usepackage{amsthm}
\usepackage{amsfonts}
\usepackage{amssymb}
\usepackage{array}  
\usepackage{graphicx} 
\usepackage{color}
\usepackage[english]{babel}
\usepackage{mathrsfs}
\usepackage{graphicx}
\usepackage{dsfont}
\definecolor{orangebis}{rgb}{0.99,0.25,0.00}
\definecolor{greenbis}{rgb}{0.10,0.85,0.10}
\definecolor{bluebis}{rgb}{0.10,0.30,0.99}
\usepackage[final]{hyperref}   
\usepackage{subcaption}

\usepackage{bbm}

\usepackage{multicol}

\usepackage{lipsum}
\usepackage{geometry}
\title{}
\date{}

\theoremstyle{plain}
\newtheorem{thm}{Theorem}[section]
\newtheorem{prop}[thm]{Proposition}
\newtheorem{lem}[thm]{Lemma}

\newtheorem{claim}[thm]{Claim}

\newtheorem{theorem}[thm]{Theorem}

\newtheorem{corollary}[thm]{Corollary}
\newtheorem{proposition}[thm]{Proposition}

\newtheorem{lemma}[thm]{Lemma}

\theoremstyle{definition}
\newtheorem{defi}[thm]{Definition}

\newtheorem{assumption}[thm]{Assumption}

\newtheorem{remark}[thm]{Remark}

\theoremstyle{remark}

\def\rm{\reversemarginpar}

\newcommand{\N}{\mathbb{N}}
\newcommand{\R}{\mathbb{R}}

\newcommand{\Z}{\mathbb{Z}}

\newcommand{\Pro}{\mathbb{P}}

\newcommand{\E}{\mathbb{E}}

\newcommand{\cross}{\text{\textup{Cross}}}

\newcommand{\var}{\text{\textup{Var}}}

\newcommand{\arm}{\text{\textup{Arm}}}

\newcommand{\ann}{\text{\textup{Ann}}}

\newcommand{\multicross}{\text{\textup{MultiCross}}}

\renewcommand{\Z}{\mathbb{Z}}
\renewcommand{\N}{\mathbb{N}}

\newcommand{\un}{\mathds{1}}

\newcommand{\grandO}[1]{O\mathopen{}\left(#1\right)}

\renewcommand{\textbf}[1]{\begingroup\bfseries\mathversion{bold}#1\endgroup}
\def\calA{\mathcal{A}}

\def\calD{\mathcal{D}}

\def\var{\mathop{\mathrm{Var}}}

\def\dist{\mathrm{dist}}

\def\E{\mathbb{E}} 

\def \eps {\varepsilon}

\def\<#1{\langle #1\rangle}

\def\bi{\begin{itemize}}  
\def\ei{\end{itemize}}
\def\bnum{\begin{enumerate}} 
\def\enum{\end{enumerate}}

\newcommand \id{\mathbbm 1}

\numberwithin{equation}{section}

\let\oldtocsection=\tocsection
\let\oldtocsubsection=\tocsubsection
\let\oldtocsubsubsection=\tocsubsubsection
\renewcommand{\tocsection}[2]{\hspace{0em}\oldtocsection{#1}{#2}}
\renewcommand{\tocsubsection}[2]{\hspace{1em}\oldtocsubsection{#1}{#2}}
\renewcommand{\tocsubsubsection}[2]{\hspace{2em}\oldtocsubsubsection{#1}{#2}}

\usepackage[foot]{amsaddr}

\begin{document}

\title[The sharp phase transition for smooth planar Gaussian fields]{The sharp phase transition for level set percolation\\ of smooth planar Gaussian fields}
\author{Stephen Muirhead}
\address{Department of Mathematics, King's College London. Current address: School of Mathematical Sciences, Queen Mary University of London}
\email{s.muirhead@qmul.ac.uk}
\author{Hugo Vanneuville}
\address{Univ. Lyon 1, UMR5208, Institut Camille Jordan, 69100 Villeurbanne, France}
\email{vanneuville@math.univ-lyon1.fr}
\date{\today}
\thanks{The first author was supported by the Engineering and Physical Sciences Research Council (EPSRC) Grant EP/N009436/1 ``The many faces of random characteristic polynomials''. The second author was supported by the ERC grant Liko No. 676999.}

\begin{abstract}
We prove that the connectivity of the level sets of a wide class of smooth centred planar Gaussian fields exhibits a phase transition at the zero level that is analogous to the phase transition in Bernoulli percolation. In addition to symmetry, positivity and regularity conditions, we assume only that correlations decay polynomially with exponent larger than two -- roughly equivalent to the integrability of the covariance kernel -- whereas previously the phase transition was only known in the case of the Bargmann-Fock covariance kernel which decays super-exponentially. We also prove that the phase transition is \textit{sharp}, demonstrating, without any further assumption on the decay of correlations, that in the sub-critical regime crossing probabilities decay exponentially.

Key to our methods is the white-noise representation of a Gaussian field; we use this on the one hand to prove new quasi-independence results, inspired by the notion of influence from Boolean functions, and on the other hand to establish sharp thresholds via the OSSS inequality for i.i.d.\ random variables, following the recent approach of Duminil-Copin, Raoufi and Tassion.\end{abstract}

\maketitle

\tableofcontents

\section{Introduction}

In recent years the strong links between the geometry of smooth planar Gaussian fields and percolation have become increasingly apparent, and it is now believed that the connectivity of the level sets of a wide class of smooth, stationary planar Gaussian fields exhibits a sharp phase transition that is analogous to the phase transition in, for instance, Bernoulli percolation. 

\medskip
To discuss these links more precisely, let us fix notation. Let $f$ be a stationary, centred, continuous Gaussian field on $\mathbb{R}^2$
with covariance kernel 
\[   \kappa(x) = \mathbb{E}[  f(0) f(x)  ]  \ , \quad x \in \mathbb{R}^2 .\]
The level sets and (upper-)excursion sets of $f$ will be denoted
\[  \mathcal{L}_\ell = \{ x :  f(x) = -\ell\} \quad \text{and}  \quad \mathcal{E}_\ell = \{ x :  f(x) \geq -\ell\}  \ , \quad \ell \in \mathbb{R} ; \]
the use of $-\ell$ instead of $\ell$ in these definitions is solely for convenience -- in particular~$\mathcal{E}_\ell$ is then increasing in both $f$ and $\ell$ -- and makes no difference to the content of our results. 

\medskip

In percolation theory, one is interested in the geometry of macroscopic components in random sets. 
A question of major interest is the existence of an unbounded connected component (when such a component exists, one says that the random set \textit{percolates}). By analogy with other planar percolation models, it is natural to expect that the excursion sets of planar Gaussian fields exhibit a phase transition at the critical level $\ell_c=0$ (since $f$ is centred, $\ell = 0$ is the `self-dual' point). More precisely, if $f$ is ergodic one expects the following phase transition at criticality:

\begin{itemize}
\item If $\ell \leq 0$, then almost surely the connected components of $\mathcal{E}_{\ell}$ are bounded;
\item If $\ell > 0$, then almost surely $\mathcal{E}_\ell$ has a unique unbounded connected component.
\end{itemize}

A rough analogy is that of water flooding the infinite landscape formed by the graph of $f$: if $\ell < 0$ then the landscape consists of an infinite landmass that contains lakes, whereas if $\ell > 0$ then instead it consists of islands surrounded by an infinite ocean. See Figure \ref{f.perc} for a simulation of the excursion sets of a stationary planar Gaussian field at (i) the zero level, and (ii) a level slightly above zero, illustrating the dramatic change in the connectivity.

\begin{figure}[h!]
\centering
\begin{subfigure}[t]{0.45\textwidth}
\resizebox{\linewidth}{!}{
\includegraphics[scale=0.6]{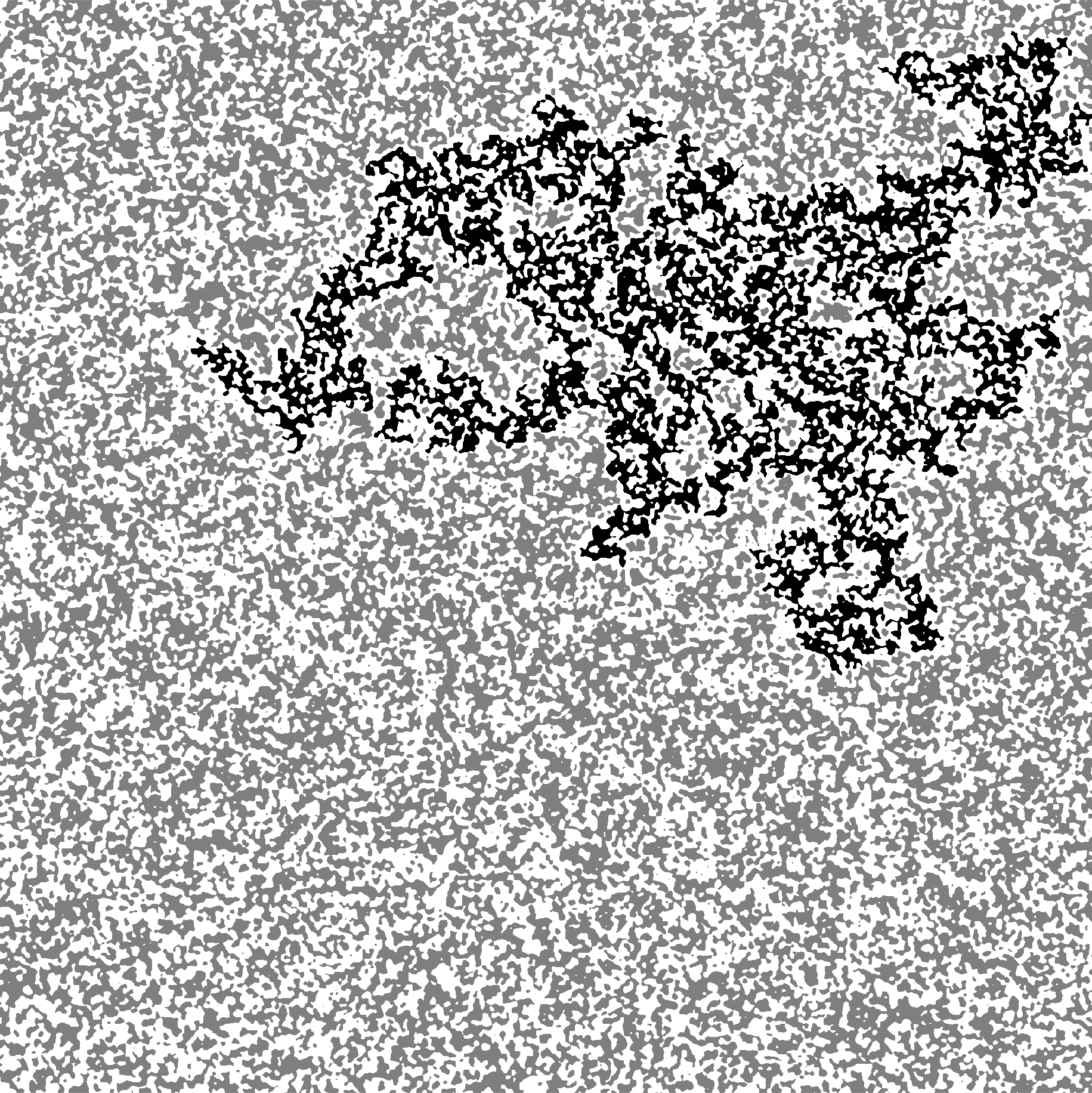}}
\end{subfigure}
\hspace{0.05cm}
\begin{subfigure}[t]{0.45\textwidth}
\resizebox{\linewidth}{!}{
\includegraphics[scale=0.6]{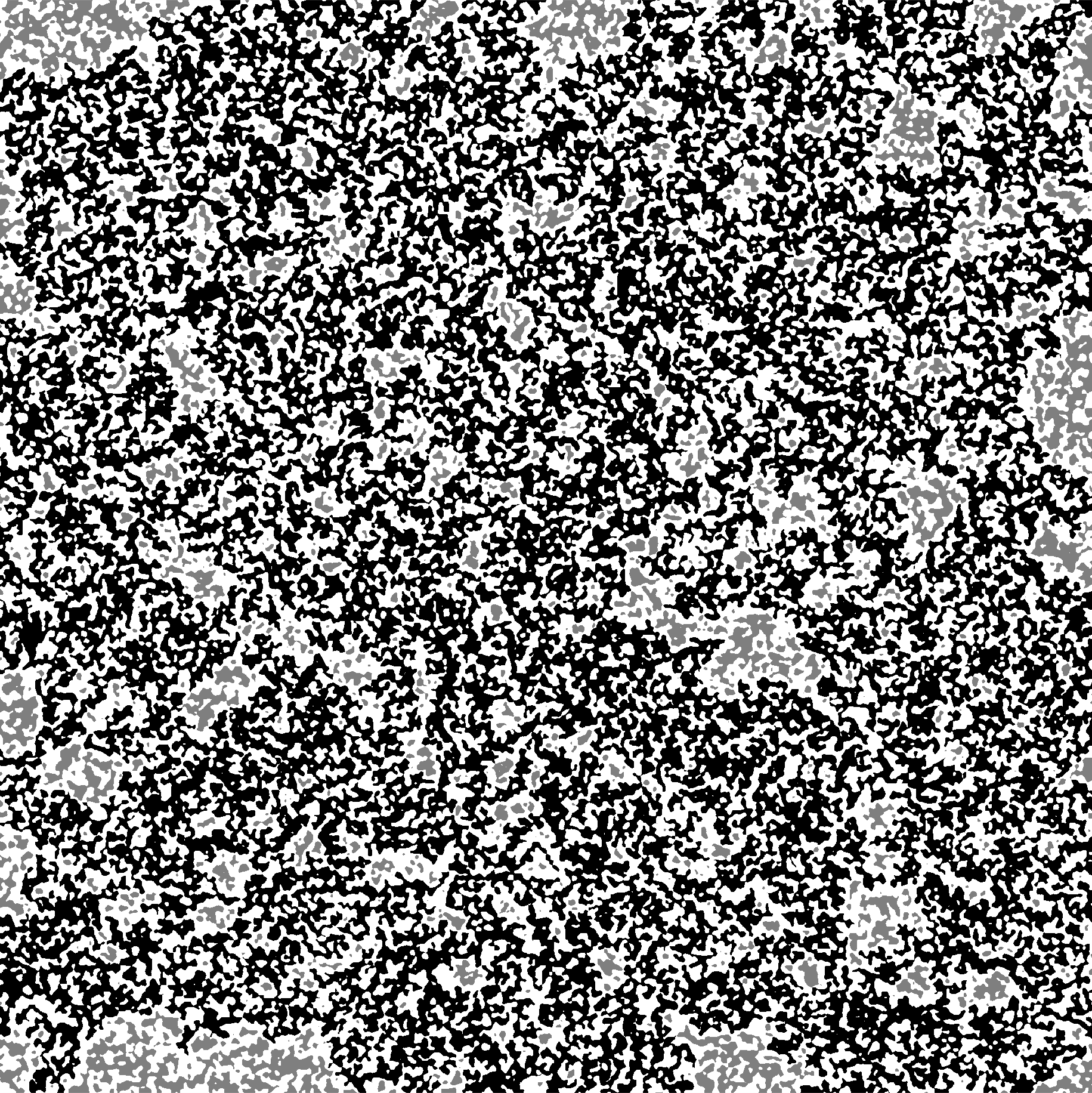}}
\end{subfigure}
\caption{A simulation of the excursion set $\mathcal{E}_\ell$ of the \textit{Bargmann-Fock} field restricted to a large square (in grey) at (i) the zero level $\ell = 0$ (left figure), at (ii) the level $\ell = 0.1$ (right figure), with the connected component of greatest area distinguished (in black). The Bargmann-Fock field is the stationary, centred Gaussian field with covariance kernel $\kappa(x) = e^{-|x|^2/2}$.  Credit: Dmitry Beliaev.}\label{f.perc}
\end{figure}

The primary aim of this paper is to establish, under mild conditions on $f$, the existence of such a phase transition at the zero level. This is the analogue, for smooth planar Gaussian fields, of the celebrated result of Kesten \cite{kesten1980critical} establishing the phase transition for (Bernoulli) bond percolation on the square lattice.

\medskip
We further establish a quantitative description of the phase transition, demonstrating that it is \textit{sharp}. More precisely, if $\ell < 0$ we show that crossing probabilities decay exponentially fast on large scales, and we also show that the `near-critical window' of levels $\ell$ for which crossing probabilities are bounded decays polynomially in the scale.

\subsection{Gaussian fields and percolation}
\label{s.litrev}
Early works to consider rigorously the connections between the geometry of planar Gaussian fields and percolation focused mainly on (i) the zero level $\ell = 0$, and (ii) very high levels $\ell \gg 1$. To the best of our knowledge the first such works were \cite{molchanov1983percolationi, molchanov1983percolationii,molchanov1986percolationiii}, in which it was shown that if $\kappa$ is sufficiently smooth and absolutely integrable then there exists a $\ell^\ast \in \mathbb{R}$ such that almost surely there is an unbounded component in $\mathcal{E}_\ell$ at every level $\ell \ge \ell^\ast$. Later it was shown \cite{alex_96}, using very different techniques, that under the assumptions of ergodicity (implied by the absolute integrability of~$\kappa$) and positive correlations (i.e.\ $\kappa \ge 0$) the level sets never percolate, i.e.\ almost surely there is no unbounded component in $\mathcal{L}_\ell$ for any $\ell \in \mathbb{R}$.

\medskip
Given these results, and by analogy with Bernoulli percolation (see below), it was natural to expect that under mild conditions (for instance if $\kappa$ is positive, integrable and sufficiently smooth) the connectivity of the level sets undergoes a phase transition at the zero level as described above. In \cite{rv_bf} this conjecture was established for the \textit{Bargmann-Fock} field, whose covariance kernel is a Gaussian function and in particular decays extremely rapidly. In this work the exact form of the covariance kernel was crucial, since the proof required explicit Fourier-type calculations to be performed.\footnote{More precisely, in~\cite{rv_bf} the phase transition at the zero level was established for the continuous Bargmann-Fock field, as well as for discrete, stationary, positively correlated Gaussian fields whose covariance decays polynomially with exponent $\beta > 4$.} As mentioned above, one of the main results of the present paper is to establish the conjecture under mild conditions on the covariance kernel; indeed we use only slightly stronger conditions than those mentioned above (to be precise, we require a `strong' positivity assumption, sufficiently symmetry, polynomial decay of correlations with exponent $\beta > 2$, and sufficient smoothness).

\medskip
In classical percolation theory the phase transition is often described in a more quantitative manner. For example, if the percolation model is critical then so-called \textit{Russo-Seymour-Welsh estimates} (RSW) hold, which are bounds on the probability that a domain is crossed that are uniform in the scale of the domain. On the other hand, if the model is sub-critical then crossing probabilities decay exponentially in the scale of the domain -- this is often referred to as the \textit{sharpness} of the phase transition.

\medskip
Recently, such statements have been proven also for the level sets and excursion sets of smooth planar Gaussian fields. The pioneering work was \cite{bg_16} in which it was shown that, assuming correlations decay polynomially with exponent at least $\beta \approx 325$, both $\mathcal{L}_0$ and $\mathcal{E}_0$ satisfy equivalents of the RSW estimates. Although the necessary decay assumptions on $\kappa$ have been progressively weakened \cite{bm_17, bmw_17, rv_rsw}, the state of the art still requires correlations to decay polynomially with exponent $\beta > 4$, much faster than that implied by the mere integrability of the covariance kernel (which corresponds roughly to $\beta > 2$). For sub-critical levels~$\ell < 0$,~\cite{rv_bf} established the exponential decay of crossing probabilities in the special case of the Bargmann-Fock field.

\medskip
A secondary aim of this paper is to establish quantitative descriptions of the phase transition. In particular, working under the same mild conditions as mentioned above (in particular, under the assumption $\beta > 2$), we show that if $\ell < 0$ then domains are crossed by $\mathcal{E}_\ell$ with probability decaying exponentially in the scale of the domain, whereas if $\ell = 0$ then RSW estimates hold. We remark that, although our results are inspired by the works mentioned above, our methods are completely independent and quite distinct (except for the use of Tassion's method~\cite{tassion2014crossing}); see Section~\ref{s.overview} for an overview of our methods, and how they relate to previous work.

\subsection{Statement of the main results}
\label{s.mr}

Recall that $f$ denotes a stationary, centred continuous planar Gaussian field with covariance kernel $\kappa$ (we exclude the degenerate case $f \equiv 0$). To state the additional assumptions that we impose on $f$ we introduce the \textit{spectral measure} $\mu$, defined as the Fourier transform of the covariance kernel $\kappa$:
\[   \kappa(x) =  \int e^{2 \pi i \langle x , s \rangle} \, d \mu(s)  ; \]
since $f$ is continuous, such a measure exists by Bochner's theorem, and satisfies $\mu(-A) = \mu(A)$, for all Borel sets $A$, and $\int d\mu = \kappa(0)$. 

\medskip 
Henceforth we shall always work under the assumption that $\mu$ is absolutely continuous with respect to the Lebesgue measure; we denote by $\rho^2$ the density of $\mu$, and refer to this as the \textit{spectral density}. Note that $\rho \in L^2(\mathbb{R}^2)$ since $\|\rho\|_{L^2} = \int d \mu = \kappa(0) \in (0,\infty)$.

\medskip
The existence of the spectral density $\rho^2$ guarantees that $f$ is ergodic \cite[Appendix B]{nazarov2015asymptotic}, and also that $\kappa(x) \to 0$ as $|x| \to \infty$ (by the Riemann-Lebesgue lemma). On the other hand, in the context of previous results this assumption is not so strong, being for instance weaker than the condition that the covariance kernel $\kappa$ is absolutely integrable (and so also weaker than the condition that correlations decay polynomially with exponent $\beta > 2$), a key assumption in previous works.

\medskip
The existence of the spectral density is fundamental for our analysis because it is equivalent to the existence of the \textit{white-noise representation} of $f$, i.e.\ the fact that
\begin{equation}
\label{e.wnr}
  f \stackrel{d}{=} q \star W 
  \end{equation}
for $q \in L^2(\mathbb{R}^2)$ satisfying $q(-x) = q(x)$, where $\star$ denotes convolution and $W$ is a planar white-noise; we give more details on this representation in Sections \ref{s.overview} and \ref{s.cond} below. To relate~\eqref{e.wnr} to the existence of the spectral density $\rho^2$, note that we can define
\begin{equation}
\label{e.q}
 q = \mathcal{F}[\rho]
 \end{equation}
where $\mathcal{F}[\cdot]$ denotes the Fourier transform; it is simple to check that $q$ possesses the required properties, namely that $q(-x) =  \mathcal{F}[\rho](-x) =  \mathcal{F}[\rho](x) = q(x)$, and $\|q\|_{L^2} = \|\rho\|_{L^2} = \| \rho^2 \|_{L^1}$.

\medskip

Conversely, if $q(-x)=q(x)$ and $q \in L^2$, then we can define $f=q \star W$ and $\rho = \mathcal{F}[q]$, which ensures that $f$ is a stationary, centred planar Gaussian field with spectral density $\rho^2$ and covariance kernel\footnote{Indeed, we use the following definition of the planar white noise: $W$ is a centred Gaussian field indexed by $L^2(\R^2)$ such that $\E \left[ \int q_1(y) dW(y) \int q_2(y) dW(y) \right] = \int q_1(y) q_2(y) dy$.}
 \[  \kappa = \mathcal{F}[\rho^2] = \mathcal{F}[\rho \cdot \rho ] = \mathcal{F}[\rho] \star  \mathcal{F}[\rho]  = q \star q. \]

\medskip 
Henceforth it will be convenient to work with \eqref{e.wnr} as the \textit{definition} of $f$ and with $\rho = \mathcal{F}[q]$ as the definition of $\rho$. To ensure $f$ enjoys some additional properties, we need to impose certain regularity conditions on $q$, which we take to hold throughout the paper:

\begin{assumption}[Regularity] \label{a.reg}
The function $q$ is in $L^2$, and for every $x \in \R^2$, $q(-x) = q(x)$. Moreover, $q$ is $C^3$ and there exist $\eps,c > 0$ such that, for every multi-index~$\alpha$ such that $|\alpha| \leq 3$, $|\partial^\alpha q(x)| \leq c |x|^{-(1+\eps)}$. Finally, the support of $\rho = F[q]$ contains an open set.
\end{assumption}

We collect important consequences of Assumption \ref{a.reg} in~Section~\ref{s.cond}. Here we simply mention that this assumption ensures (by dominated convergence) that $\kappa=q \star q$ is $C^6$ which permits us to define $f$ as a continuous (in fact $C^2$) modification of $q \star W$, rather than $q \star W$ itself, and also guarantees that, for each $\ell \in \mathbb{R}$, the level set $\mathcal{L}_\ell$ almost surely consists of a collection of simple curves.

\begin{remark}
In addition to $f$, we shall also consider in this paper \textit{uncountable} families of Gaussian fields $f_r$, $f_r^\eps$ and $f^\eps$, indexed by $r \geq 1$ and $\eps > 0$, constructed using the white-noise~$W$. Although the existence of \textit{simultaneous} modifications of these fields such that they are \textit{all} continuous is not obvious, this is not essential for us since we will only ever consider either: i) fixed parameters $r$ and $\eps$; or ii) limits as $r \to \infty$ or $\eps \to 0$ in order to deduce a result for $f$, for which we can always work with countable subsequences.
\end{remark}

For our main results to hold, we shall need some or all of the following additional assumptions on $q$:

\begin{assumption}[$D_4$ Symmetry]  \label{a.sym}
The function $q$ is symmetric under both (i) reflection in the $x$-axis, and (ii) rotation by $\pi/2$ about the origin.
\end{assumption}

\begin{assumption}[Weak or strong positivity] \label{a.swp}  \ 
\begin{enumerate}
\item (Weak positivity) $\kappa = q \star q \ge 0$;
\item (Strong positivity) $q \ge 0$.
\end{enumerate}
\end{assumption}

\begin{assumption}[Polynomial decay of correlations; depends on a parameter $\beta>0$]  \label{a.dec} 
There exists a constant $c > 0$ such that, for every $|x| > 1$ and multi-index $\alpha$ such that $|\alpha| \le 1$,
\begin{equation} \label{e.beta}
   | \partial^\alpha q(x)  |  < c |x|^{- \beta }  .
   \end{equation}
\end{assumption}

\medskip 
We emphasise that the decay condition \eqref{e.beta} is a slightly stronger version of the assumption, appearing in previous works, that $\kappa(x) = O(|x|^{-\beta})$. Moreover, $q(x) = O(|x|^{-\beta})$ is equivalent to $\kappa(x) = O(|x|^{-\beta})$ for $\beta > 2$ if $q \ge 0$ is also assumed. Observe also that Assumption~\ref{a.reg} implies that~\eqref{e.beta} holds for some $\beta > 1$, and on the other hand, if~\eqref{e.beta} holds for $\beta > 2$ and $q$ is not identically equal to $0$, then the support of $\rho = F[q]$ contains an open set since it is continuous.

\medskip
Although we have chosen to state Assumptions \ref{a.sym} and \ref{a.swp} in terms of $q$, they have natural analogues for the spectral measure $\mu$. First, the weak positivity condition in Assumption~\ref{a.swp} is equivalent to the spectral density $\rho^2$ being positive-definite, whereas strong positivity is equivalent to $\rho$ being positive-definite. Second, Assumption \ref{a.sym} is equivalent to any of $\rho$, $\mu$, $\kappa$ or the law of $f$ satisfying the same symmetries. We remark also that sufficient conditions for Assumptions~\ref{a.reg} and~\ref{a.dec} could be given in terms of the spectral density $\rho^2$ using classical results from Fourier analysis.

\medskip
We are now ready to state our first theorem, establishing the phase transition at the zero level under the above conditions:

\begin{theorem}[The phase transition at the zero level]
\label{t.main1}
Suppose that Assumptions \ref{a.reg} and \ref{a.sym} hold, that the strong positivity condition in Assumption~\ref{a.swp} holds, and also that Assumption \ref{a.dec} holds for a given $\beta > 2$. Then the following are true:
\begin{itemize}
\item If $\ell \leq 0$, then almost surely the connected components of $\mathcal{E}_{\ell}$ are bounded;
\item If $\ell > 0$, then almost surely $\mathcal{E}_\ell$ has a unique unbounded connected component.
\end{itemize}
\end{theorem}

Since $f = -f$ in law, the same result holds for $\mathcal{E}_\ell^c$, the complement of the excursion set. We can also state a version of the result for `thickenings' of the zero level set, i.e.\ $\mathcal{L}_0^\eps = \{ x : |f(x)| \le \varepsilon \}$:

\begin{theorem}[The phase transition for thickened zero level sets]
\label{t.main1a}
Under the same conditions as in Theorem \ref{t.main1}: \begin{itemize}
\item If $\eps =  0$, then almost surely the connected components of $\mathcal{L}^{\eps}_0$ are bounded;
\item If $\eps > 0$, then almost surely $\mathcal{L}^\eps_0$ has a unique unbounded connected component.
\end{itemize}
\end{theorem}

\begin{remark}
\label{r.kesten}
Our assumptions are stronger than in the early works \cite{molchanov1983percolationi, molchanov1983percolationii, alex_96} described in Section \ref{s.litrev}, and so Theorem \ref{t.main1} was already known for $\ell \le 0$ and for $\ell$ sufficiently large. What is new is the statement that $\mathcal{E}_\ell$ percolates at \textit{every} positive level $\ell > 0$, which was previously only known in the case of the Bargmann-Fock field \cite{rv_bf}. 

\medskip
This statement is the analogue, for smooth planar Gaussian fields, of Kesten's celebrated result for (Bernoulli) bond percolation on the square lattice, which we recall now. Fix $p \in [0,1]$ and colour each edge of the square lattice $\Z^2$ black independently with probability~$p$. Harris \cite{harris1960lower} showed that at the `self-dual' point, $p = 1/2$, there are almost surely no unbounded black clusters (i.e.\ unbounded components of the sub-graph of black edges). Much later, Kesten famously proved \cite{kesten1980critical} the existence of a phase transition at the `self-dual' point, i.e.\ showed that if $p>1/2$ then almost surely there is a (unique) unbounded black cluster. 
\end{remark}

\begin{remark}
Theorems \ref{t.main1} and \ref{t.main1a} require the strong positivity condition $q \ge 0$, which can be contrasted with previous works (see the discussion in Section \ref{s.litrev}) that assumed only weak positivity~$\kappa \ge 0$. We do not believe strong (as opposed to weak) positivity to be fundamental to the result, and in fact we suspect it could be removed (at the expense of the full strength of Theorem \ref{t.main3} below), although we do not pursue this here. In fact, strong positivity is only used at a single place in the proof (see the discussion in Section~\ref{s.connect}).

\medskip
We also remark that an (apparently) even stronger positivity condition (often called \textit{total positivity}) holds for the discrete planar Gaussian free field (GFF), and was a crucial ingredient in recently obtained results that bear some similarities to ours \cite{rodriguez20150} (although in a very different setting).
\end{remark}

\begin{remark}
Our techniques and results likely extend to the setting of sequences of smooth Gaussian fields on compact manifolds such as the sphere or flat torus, as in \cite{bmw_17}, although additional technical difficulties may arise. Similar results likely hold also for many classes of non-Gaussian fields (e.g.\ chi-squared fields, shot-noise etc.), which could have potential applications in physics \cite{w_82} and in the statistical testing of spatial noise \cite{bel_17}. However, many of our techniques are tailored to the Gaussian setting, so would not immediately apply to other classes of fields.
\end{remark}

We next turn to a quantitative description of the phase transition, and show in particular that it is \textit{sharp}. For this we need to introduce notation for \textit{crossing events}. For each $r > 0$ and $x \in \mathbb{R}^2$, let $B_r(x) = \{ y \in \mathbb{R}^2 : |x-y| \le r\}$ denote the Euclidean ball with radius $r$ centred at~$x$, and abbreviate $B_r=B_r(0)$. Define a \textit{quad} $Q$ to be a simply-connected piece-wise smooth compact domain $D \subset \mathbb{R}^2$ together with two disjoint boundary arcs $\gamma$ and $\gamma'$. One can take, for instance,~$D$ to be a rectangle and $\gamma$ and $\gamma'$ to be opposite edges. 

\medskip
For each quad $Q$ and level $\ell$, let $\cross_\ell(Q)$ denote the event that there is a connected component of $\mathcal{E}_\ell$ that crosses $Q$, i.e., whose intersection with~$Q$ intersects both $\gamma$ and $\gamma'$. Similarly, for $0 < r_1 < r_2$, let $\arm_\ell(r_1,r_2)$ denote the event that there is a connected component of $\mathcal{E}_\ell$ that intersects both $\partial B_{r_1}$ and $\partial B_{r_2}$. Note that our assumptions on $f$ and $Q$ ensure that each of these events is measurable.

\begin{thm}[Sharpness of the phase transition]
\label{t.main2}
Suppose that Assumptions \ref{a.reg} and \ref{a.sym} hold, that the weak positivity condition in Assumption \ref{a.swp} holds, and that Assumption~\ref{a.dec} holds for a given $\beta > 2$. Then for every quad $Q$, 
\[     \inf_{R > 0} \, \mathbb{P} [ f \in \cross_0(RQ) ]   > 0  \quad \text{and} \quad    \sup_{R > 0}  \, \mathbb{P} [f \in \cross_0(RQ)   ]  < 1,  \]
and moreover there exist $c_1, c_2 > 0$ such that, for each $0 < r < R$,
\[   \mathbb{P}[   f   \in \arm_0(r, R ) ]     <  c_1 \left( \frac{r}{R} \right)^{c_2}.  \] 
Suppose in addition that the strong positivity condition in Assumption \ref{a.swp} holds. Then the following are true:
\begin{itemize}
\item If $\ell < 0$, then for every quad $Q$ there exist $c_1, c_2 > 0$ such that, for all $R \ge 1$,
\begin{equation}
\label{e.l<0}
 \Pro \left[ f \in \cross_\ell(RQ ) \right] < c_1 e^{-c_2s}     . 
 \end{equation}
\item If $\ell > 0$, then for every quad $Q$ there exist $c_1, c_2 > 0$ such that, for all $R \ge 1$,
\begin{equation}
\label{e.l>0}
 \Pro \left[ f \in \cross_\ell(RQ ) \right] > 1- c_1 e^{-c_2s}    . 
 \end{equation}
\end{itemize}
\end{thm}

\begin{remark}
The first statement of Theorem \ref{t.main2} gives analogues of the RSW estimates in critical percolation theory (see for instance~\cite{grimmett1999percolation,bollobas2006percolation}), which have previously been established under stronger conditions on the decay of correlations (roughly corresponding to $\beta > 4$) \cite{bg_16, bm_17, rv_rsw}. The statement about $ \arm_0(r, R)$  is the analogue of the \textit{one-arm decay} in percolation theory, and follows in a straightforward way from the RSW estimates (at least, if a preliminary `quasi-independence' property has been established; see Theorem \ref{t.qi}). Notably,  we need only the weak positivity condition $\kappa \ge 0$ for these statements.
\end{remark}

\begin{remark}
The second and third statements of Theorem \ref{t.main2} give quantitative bounds on crossing probabilities in `non-critical' regimes; previously such bounds had only been established for the Bargmann-Fock field \cite{rv_bf}. 

\medskip
In the case $\ell < 0$, the statement is a bound on the decay of crossing probabilities in the sub-critical regime, showing in particular that sub-critical crossing probabilities decay exponentially, just as for Bernoulli percolation \cite{kesten1980critical}.

\medskip
In the case $\ell > 0$, the second statement is a quantitative description of the claim in Theorem \ref{t.main1} that~$\mathcal{E}_\ell$ percolates, and in fact we use this statement to infer the percolation of $\mathcal{E}_\ell$ via a simple Borel-Cantelli argument. 
\end{remark}

\begin{remark}
Similar results to those in Theorem \ref{t.main2} could also be deduced for the level sets~$\mathcal{L}_\ell$ -- i.e.\ defining the crossing events $\cross_\ell(RQ )$ relative to $\mathcal{L}_\ell$ rather than $\mathcal{E}_\ell$ -- but we have chosen to omit such results. A notable exception is \eqref{e.l>0}, which does not hold for $\mathcal{L}_\ell$.
\end{remark}

\medskip

One consequence of Theorem \ref{t.main2} is that, for each $\ell >0$ and quad $Q$,
\begin{equation}
\label{e.window}
\Pro \left[ f \in \cross_\ell(RQ) \right] \to 1 \quad \text{as } R \to \infty.
\end{equation}
A natural question is to determine how slowly a positive sequence $\ell_R \to 0$ must decay in order to ensure that \eqref{e.window} still holds for $\ell = \ell_R$; in other words, to quantify the size of the \textit{near-critical window}. Our next result shows that this window is of polynomial size, as it is for Bernoulli percolation. 

\begin{theorem}[Polynomial bounds on the near-critical window]\label{t.main3}
Suppose that Assumptions~\ref{a.reg} and \ref{a.sym} hold, that the weak positivity condition in Assumption~\ref{a.swp} holds, and that Assumption~\ref{a.dec} holds for a given $\beta > 2$. Then for each $c_1 > 1$ and every quad $Q$,
\[ \limsup_{R \to \infty} \, \Pro \left[ f \in \cross_{R^{-c_1} }(RQ) \right] < 1 . \]
Suppose in addition that the strong positivity condition in Assumption~\ref{a.swp} holds. Then there exists a $c_2 > 0$ such that, for every quad $Q$,
\[ \lim_{R \to \infty} \Pro \left[ f \in \cross_{R^{-c_2}}(RQ ) \right] = 1  .\]
\end{theorem}

\begin{remark}
Again by analogy with Bernoulli percolation, it is natural to conjecture that the near-critical window is of polynomial size with exponent exactly $3/4$, i.e.\ the conclusion of Theorem~\ref{t.main3} is true for every $0 < c_2 < 3/4 < c_1 $. Our result shows that the exponent is strictly positive and at most $1$. This is comparable to what is known for bond percolation on the square lattice, for which the exponent has been shown to be strictly positive and \textit{strictly} less than~$1$, see~\cite{kesten1987scaling} (on the other hand, for site percolation on the triangular lattice the conjecture is known in full~\cite{Smirnov2001critical}).
\end{remark}

\subsection{A family of examples}
\label{s.e}
In this section we introduce a family of smooth planar Gaussian fields that illustrates the generality and scope of our results.

\medskip
Consider the \textit{rational quadratic} kernel (sometimes also called the \textit{Student-}$t$ kernel)
\[  \text{RQ}_\beta(x)  = (1 + |x|^2)^{-\beta/2}  , \quad \beta > 0 , \]
which is continuous, isotropic and positive-definite on $\mathbb{R}^2$; this kernel is extensively used in the modelling of spatial data, see, e.g., \cite[Chapter 4]{rw_06}. 

\medskip Fix a value of the parameter $\beta > 2$ and let $q = \text{RQ}_\beta$, which satisfies the necessary conditions for the white-noise representation to be valid.  The Fourier transform of $q$ is known \cite[Chapter 17]{poularikas99}, and is given by 
\[ \rho(x) = \mathcal{F}[ q ](x) \propto |x|^{\beta/2 - 1} K_{\beta/2 - 1}(|x|)  ,\]
 where $K_n(z)$ is the modified Bessel function of the second kind. Hence, $f = q \star W$ is a stationary planar Gaussian field with spectral density 
 \[ \rho^2(x) = (\mathcal{F}[ q ])^2(x) \propto |x|^{\beta - 2} K^2_{\beta/2 - 1}(|x|)   .\]
Observe that each of Assumptions \ref{a.reg} and \ref{a.sym}, and the strong positivity condition in Assumption~\ref{a.swp}, are easily verified. Moreover, one can check that the decay condition in Assumption~\ref{a.dec} is satisfied for $\beta$.
 
  \medskip
Given the discussion above, we see that our results apply to this planar Gaussian field for each $\beta > 2$. On the other hand, the covariance kernel of this field satisfies
\[   \kappa(x) \sim c |x|^{-\beta}  \ , \quad \text{as } |x| \to \infty, \]
and so none of the previous results in the literature (for instance in \cite{bg_16, rv_rsw}) apply to this field unless $\beta > 4$, and even then only the RSW estimates of Theorem~\ref{t.main2} were known~\cite{rv_rsw}. In particular, the results in Theorems \ref{t.main1} and \ref{t.main2} in the case $\ell > 0$, and the result in Theorem~\ref{t.main3}, were not previously known for \textit{any} value of $\beta > 2$.

\subsection{The percolation universality class}

A major unresolved question raised by our work is to determine how rapidly correlations must decay in order for the connectivity of the level sets of a smooth Gaussian field to be well-described on large scales by Bernoulli percolation, i.e.\ to determine the boundary of the `percolation universality class'. 

\medskip
In the physics literature, the `Harris criterion' (see, e.g., \cite{weinrib1984long}) is a well-known heuristic that determines whether long-range correlations influence large scale properties of discrete percolation models. Translated to our setting, the criterion suggests that the percolation universality class is determined by the convergence of
\[  \frac{1}{R^{5/2}}  \int_{x \in B_R} \int_{y \in B_R} \kappa(x-y) \, dx dy ,\]
which is roughly equivalent to demanding that $\kappa$ has polynomial decay with exponent $\beta > 3/2$ in the case $\kappa \geq 0$ (compared to $\beta > 2$ that is required in our results). It is an interesting question whether the Harris criterion can be formalised into a rigorous description of the universality class.

\medskip
Further, the analogy with Bernoulli percolation should go beyond even the results in Theorems~\ref{t.main1}, \ref{t.main2} and \ref{t.main3}. For instance, one might expect that the zero level sets $\mathcal{L}_0$ should, on large scales, behave similarly to the so-called $\textit{SLE}_6$ process (or, more precisely, the $\textit{CLE}_6$ loop ensemble \cite{sheffield99}), which is conjectured to describe the scaling limit of the boundaries of clusters in Bernoulli percolation; this conjecture is (only) proved for site percolation on the regular triangular lattice \cite{smirnov2007towards,camia2007critical}. This has been conjectured for the random plane wave (RPW) \cite{bds_07}, which is a universal Gaussian model for eigenfunctions of the Laplacian on generic (i.e.\ chaotic) smooth manifolds. The RPW has correlations that decay extremely slowly -- only at rate $1/\sqrt{R}$ -- but since the correlations are highly oscillatory, the Harris criterion is nevertheless still satisfied \cite{bs_07}.
 
 \medskip
There are also certain Gaussian fields for which the level sets are known not to resemble percolation clusters on large scales. For example, it is known that the `level lines' of the planar GFF are $\textit{SLE}_4$ processes \cite{ss_99}, and so the planar GFF lies in an entirely different universality class to percolation. 

\subsection{Overview of rest of the paper}

In Section \ref{s.overview} we present an overview of our methods and give a general outline of the proof of the main results. In Section \ref{s.gauss} we collect the arguments that are particular to the Gaussian setting of our work. In Section \ref{s.qirsw} we establish the crucial `quasi-independence' property for crossing events, and use it to deduce the RSW estimates for $\ell = 0$ that comprise the first statement of Theorem~\ref{t.main2}. We also deduce from the RSW estimates the first statement of Theorem~\ref{t.main3}. In Section \ref{s.osss} we begin our study of the sharp phase transition, establishing a qualitative description of the phase transition for crossings of a fixed rectangle at large scales. Finally, in Section \ref{s.final} we bootstrap the aforementioned result to complete the proof of the main results.

\subsection{Acknowledgments}
The authors are grateful to Dmitry Beliaev and Alejandro Rivera for many fruitful discussions, and to Dmitry Beliaev in particular for suggesting Kolmogorov's theorem as a way to prove Lemma \ref{l.btis}, and also for kindly allowing the use of his simulations in Figures \ref{f.perc} and~\ref{f.algo}. Additionally the first author would like to thank Jeremiah Buckley, Naomi Feldheim, Michael McAuley and Mihail Poplavskyi, and the second author would like to thank Vincent Beffara, Charles-Édouart Bréhier, Hugo Duminil-Copin, Thibault Espinasse, Christophe Garban, Damien Gayet, Thomas Letendre, Aran Raoufi, Avelio Sep\'{u}lveda and Vincent Tassion, who all assisted with helpful discussions.

\bigskip

\section{Overview of our methods and outline of the proof}
\label{s.overview}
Compared to previous works on the links between Gaussian fields and percolation, our methods contain several new techniques that we emphasise here:
\begin{itemize}
\item First, our overarching methodology is to work \textit{directly in the continuum}, rather than restricting the Gaussian field to a lattice as was done in all previous works on the topic (with the exception of \cite{alex_96}). This opens up new techniques, such as our application of the Cameron--Martin theorem to control the effect of varying the level (this is inspired by similar arguments in the analysis of Boolean functions, see Section~\ref{s.vl}). On the other hand, we do `discretise' the field by approximating it by a random variable taking values in finite dimensional subspaces of the set of continuous planar functions. 
\item Second, we exploit heavily the white-noise representation of a Gaussian field in \eqref{e.wnr}. Even though this representation is well-known in other contexts (see, e.g., \cite{h_02}, or \cite{m_69,c_76} for a closely-related representation), as far as we know it has never been used to study the connectivity of level sets. We give more details on how we use this representation immediately below.
\item Third, we prove that there is a phase transition at level $0$ by appealing to recent advances in the study of randomised algorithms, and in particular the development of the \textit{OSSS inequality}; this is inspired by recent applications of similar ideas to various models \cite{duminil2017exponential, duminil2017sharp, duminil2018poisson,ahlberg2017noise}. Again we give more details immediately below.
\item Finally, we use a `sprinkled' quasi-independence property (see Proposition \ref{p.sqi}) in order to bootstrap the results obtained by using randomised algorithms and obtain sharpness results (i.e.\ exponential decay of connection probabilities in the subcritical phase). Here `sprinkled' means that instead of comparing $\Pro \left[ f \in A \right] \Pro \left[ f \in B \right]$ to $\Pro \left[ f \in A \cap B \right]$, we compare the first quantity to $\Pro \left[ f - \eps \in A \cap B \right]$. This is in contrast to \cite{rv_bf}, in which the rapid decay of correlations in the Bargmann-Fock field ensured that a classical notion of quasi-independence was sufficient.
\end{itemize}

In the rest of this section we (i) describe how we exploit the white-noise representation, (ii) explain the OSSS inequality and how we apply it, and (iii) give a brief outline of the proof.

\subsection{Truncation and discretisation}
\label{s.td}

The white-noise representation is useful because of two operations  -- \textit{truncation} and \textit{discretisation} -- that allow us to couple $f$ to other Gaussian fields which are close to $f$ with high probability but have desirable properties. 

\medskip 
\textbf{Truncation.} 
Recall that we take \eqref{e.wnr} to be the definition of $f$. For each $r \ge 1$, let $\chi_r: \mathbb{R}^2 \to [0, 1]$ be a smooth isotropic approximation of the radial cut-off function $\un_{| \cdot | \leq r/2}$. More precisely, we ask that $\chi_r$ is smooth, isotropic, and satisfies
\[ \chi_r(x) = \begin{cases} 1, & \text{if } |x|\leq r/2-1/4 , \\  0, & \text{if } |x| \geq r/2 , \end{cases} \]
and that the $k^\text{th}$ derivatives of $\chi_r$, for all $k \ge 1$, are uniformly bounded in $x \in \mathbb{R}^2$ and $r \ge 1$. Abbreviate $q_r = q \chi_r$, and observe that, for all $r \ge 1$, either $q_r$ is identically equal to $0$ or it satisfies Assumption~\ref{a.reg} whenever $q$ does. Hence, for each $r \ge 1$ we can define the stationary centred planar Gaussian field
\[ f_r = q_r \star W .\]
We call this the $r$-\textit{truncation of} $f$, and highlight the crucial fact that it is an $r$\textit{-dependent field}, meaning that $f_r|_{D_1}$ and $f_r|_{D_2}$ are independent for all subsets $D_1, D_2 \subseteq \mathbb{R}^2$ such that $d(D_1, D_2) \ge r$. Moreover, we have good control on the difference $f - f_r$ since, by definition,
\[   (f - f_r)(\cdot) =  ((q - q_r) \star W)(\cdot) =  \int (q - q_r)(\cdot -u)  \, dW(u) . \] 
As we show in Section \ref{s.gauss}, under Assumptions \ref{a.reg} and \ref{a.dec} we can control this difference well in the sup-norm on compact sets, which is what we shall need for our application to crossing events (see Section~\ref{s.outline} for more detail on this application).

\medskip 
\textbf{Discretisation.} 
Next, for each $\varepsilon > 0$ we couple the white-noise $W$ to a discretised version $W^\eps$ at scale $\varepsilon$ by setting
\[  \eta_v = \eps^{-1} \int_{x \in v + [-\eps/2, \eps/2]^2 } dW(x)  \ , \quad v \in \eps \Z^2 , \]
and defining 
\[ W^\eps(x)=\eps^{-1} \sum_{v \in \eps \Z^2} \eta_v \id_{x \in v + [-\eps/2, \eps/2]^2 }  . \]
Note that the $\eta_v$ are distributed as i.i.d.\ standard Gaussian variables. As we show in Proposition~\ref{p.areg}, under suitable assumptions on $q$ each $\varepsilon > 0$ gives rise to a planar Gaussian field via
\[ f^\eps=q \star W^\eps    . \]
We call this the $\varepsilon$-\textit{discretisation of} $f$, and note that it is \textit{stationary with respect to lattice shifts} $x \mapsto x + v$, $v \in \eps \Z^2$. 

\medskip
Similarly as for $f_r$, in Section \ref{s.gauss} we show how Assumption \ref{a.reg} allows us to control the sup-norm of $f - f^\eps$. To give an idea how this is done, let us rewrite the map $x \mapsto f^\eps(x)$ in a slightly different form. For each continuous function $g : \mathbb{R}^2 \to \mathbb{R}$ and each $\eps > 0$ and $x \in \R^2$, let $g^{x, \eps}$ be defined by setting, for each $v \in \eps \Z^2$ and $u \in v + [-\eps/2, \eps/2]^2$,
\begin{equation}
\label{e.pwc}
 g^{x, \eps}(x+u) = \eps^{-2} \int_{w \in  v  + [-\eps/2, \eps/2]^2} g(x+w) \, dw  .
  \end{equation}
 i.e.\ $g^{x, \eps}$ is a piece-wise constant function that takes its average value on each face of the shifted lattice $x + (\eps/2, \eps/2) + \eps \Z^2$; we call $g^{x, \eps}$ the \textit{piece-wise constant approximation of} $g$. With this definition, we can express $f^\eps(x)$ as 
\begin{align*}
 f^\eps(x) & = (q \star W^\eps)(x)  =  \eps^{-1} \sum_{v \in \eps \Z^2}   \eta_v  \int_{u \in v  + [-\eps/2, \eps/2]^2}   q(x-u) \, du  
 \\
 &= \sum_{v \in \eps \Z^2}  \bigg( \int_{u \in v  + [-\eps/2, \eps/2]^2}   dW(u)  \bigg) \bigg( \eps^{-2} \int_{u \in v + [-\eps/2, \eps/2]^2}  q(x-u) \, du   \bigg)   \\
  &= \sum_{v \in \eps \Z^2}   \int_{u \in v  + [-\eps/2, \eps/2]^2} q^{x, \eps}(x-u) \,  dW(u)        \\
 &  = \int  q^{x, \eps}(x-u) \, dW(u) =   \left( q^{x, \eps} \star W \right)(x) .
 \end{align*}
The validity of the interchange of sum and integral is established by computing
\[
\E \bigg[ \Big( \sum_{v \in ( \eps \Z^2 ) \cap [-n,n]^2 }   \int_{u \in v  + [-\eps/2, \eps/2]^2} q^{x, \eps}(x-u) \,  dW(u)  -  \int  q^{x, \eps}(x-u) \, dW(u) \Big)^2 \bigg],
\]
which goes to $0$ as $n \to \infty$ by the definition of $W$ and the fact that $q^{x, \eps}(x-\cdot) \in L^2$ (which is a direct consequence of Assumption~\ref{a.reg}). Hence, the point-wise difference between $f$ and $f^\eps$ can be expressed as
\[   (f - f^\eps)(x) =   \left( (q - q^{x, \eps}) \star W \right)(x)  .  \]

\medskip
To avoid confusion, we remark that the description of $f^\eps$ as \textit{discrete} refers to the discrete white-noise $W^\eps$ and not the field $f^\eps$ itself, which is a continuous random field. This distinguishes our discretisation procedure from the discretisation procedures used in previous works on this topic, which consisted of restricting $f$ to the vertices of a lattice \cite{bg_16, bm_17,  rv_bf, rv_rsw}. On the other hand, the field $f^\eps$ is approximately \textit{finite-dimensional}. To explain this, observe that by combining the above ideas we can define, for each $r \ge 1$ and $\varepsilon > 0$, the field
\[ f^\eps_r := q_r \star W^\eps  .\]
This field is finite-dimensional on any compact domain $D$, meaning that we can write
\[ f^\eps_r |_D = G(\eta_1, \ldots , \eta_N) \]
for a function $G$ and $N \in \mathbb{N}$, where $\eta_i$ are standard i.i.d.\ Gaussian variables. This finite-dimensional approximation of $f$ is useful for applying the OSSS inequality, which we explain next.

\subsection{The OSSS inequality} \label{s.OSSSsub}
The OSSS inequality originated in~\cite{o2005every} in the study of the complexity of algorithms; we refer to~\cite{duminil2017sharp} and~\cite[Chapter 12]{garban2014noise} for more about its origins. In the present paper we are solely interested in the recent applications of this inequality to establish sharp phase transitions in many important models of statistical physics (e.g.\ FK-cluster, Voronoi percolation, Poisson-Boolean percolation; see \cite{duminil2017exponential, duminil2017sharp, duminil2018poisson,ahlberg2017noise}).

\medskip
Let us first recall the OSSS inequality. Let $\Lambda$ be a finite set, let $\mu$ be a probability measure on a measurable space $(E,\mathcal{E})$, and consider the product probability space $(E^\Lambda,\mathcal{E}^{\otimes \Lambda},\mu^{\otimes \Lambda})$. Given any $A \in \mathcal{E}^{\otimes \Lambda}$ of this product space and a coordinate $i \in \Lambda$, the \textit{influence} $I_i^\mu(A)$ of the $i^{\rm{th}}$ coordinate on $A$ is defined as the probability that resampling the $i^{\rm{th}}$ coordinate modifies $\un_A$, i.e.,
\[
I_i^\mu(A) = \Pro \left[ \un_A(\omega) \neq \un_A(\widetilde{\omega}) \right]  ,
\]
where $\omega \sim \mu^{\otimes \Lambda}$ and where $\widetilde{\omega}=\omega$ except that $\omega_i$ is resampled independently. 

\medskip
Now, let $\mathcal{A}$ be a \textit{random algorithm} that determines $A$; this means that $\mathcal{A}$ is a procedure that reveals step-by-step the coordinates of $\omega$ and stops once the value of $\un_A(\omega)$ is known, and for which the choice of the next coordinate to be revealed depends only on (i) a \textit{random seed} that is initialised once and for all at the start of the algorithm, (ii) the coordinates that have already been revealed, (iii) the value of $\omega$ on these coordinates. The \textit{revealment} $\delta_i^\mu(\mathcal{A})$ of the $i^{\rm{th}}$ coordinate for the algorithm $\mathcal{A}$ is the probability that $\omega_i$ is revealed by the algorithm. 

\medskip
The OSSS inequality (originally proven for $E$ finite~\cite{o2005every} but which also holds in the general\footnote{In both \cite{o2005every} and \cite{duminil2017exponential} the OSSS inequality is written without randomness on the initial seed but the case of a random seed is a direct consequence of the deterministic case.} case~\cite[Remark~5]{duminil2017exponential}) can be stated as follows:
\begin{thm}[OSSS inequality]
\label{t.osss}
For every $A \in \mathcal{E}^{\otimes \Lambda}$ and algorithm $\mathcal{A}$ that determines $A$,
\[
\text{\textup{Var}}_\mu(\un_A(\omega)) \leq \sum_{i \in \Lambda} \delta^\mu_i(\mathcal{A}) I_i^\mu(A)  .
\]
\end{thm}

Let us explain briefly, in the setting of bond percolation on the square lattice, how the OSSS inequality can be used to establish the existence of a phase transition. Let $\Lambda$ be the set of edges of the lattice, and let $E=\{ 0,1 \}$, where $1$ represents the colour black. For each $R > 0$, let $\cross(R)$ denote the event that there is a black crossing of a square of size $R$. By duality properties, we know that $\mathbb{P}_{1/2}(\cross(R)) = 1/2$ for each $R > 0$, where $\mathbb{P}_p$ is the measure associated to the model with probability $p$. Once this has been observed, the first step in the proof of the phase transition is \textit{Russo's formula} (see for instance~\cite{grimmett1999percolation,bollobas2006percolation}) which implies that
\begin{equation}
\label{e.rus1}
\frac{d}{dp} \Pro_p \left[ \cross(R) \right] = \frac{1}{2p(1-p)}\sum_{i \in \Lambda} I_i^{\mu_p}(\cross(R)) ,
\end{equation}
where $\mu_p=p\delta_1+(1-p)\delta_{-1}$. Note that the factor $1/(2p(1-p))$ arises since the notion of influence that one would usually use in order to write Russo's formula in the Boolean case $E=\{ 0,1\}$ is 
\[ \Pro \left[ \un_A(\omega) \neq \un_A(\widehat{\omega}) \right] \]
 where $\widehat{\omega}=\omega$ except that we change (and \textit{not} resample) the $i^{\rm{th}}$ coordinate. Applying the OSSS inequality, we have
\[
\frac{d}{dp} \Pro_p \left[ \cross(R) \right] \geq \frac{1}{2p(1-p)}\frac{\var_{\mu_p}(\un_\cross(R))}{\sup_{i} \delta^{\mu_p}_i(\cross(R) )} .
\]
Hence, in order to demonstrate a phase transition, for instance to show that
\[  \frac{d}{dp} \Pro_p \left[ \cross(R) \right]  \bigg|_{p = 1/2} \to \infty  \ , \quad \text{as } R \to \infty, \]
it is sufficient to show that 
\[ \sup_{i} \delta^{\mu_{1/2}}_i(\cross(R) ) \to 0  , \quad \text{as } R \to \infty. \]
This latter step can be done by noting that, if the algorithm is suitably chosen, the revealments can be controlled by the probability of the one-arm event, i.e.\ the event that there is a black path from $0$ to distance $> c R$ (see~\cite{benjamini1999noise} and \cite{schramm2010quantitative} where such algorithms are used to study noise sensitivity questions). In turn, this is precisely what can be deduced from the RSW estimates at the critical level $p=1/2$.

\medskip
In the above argument the OSSS inequality was applied to crossing events of squares, whereas the control of the revealments $\delta_i$ was achieved by analysing one-arm events; this is roughly how we shall apply the OSSS inequality (see also \cite{ahlberg2017noise} where this argument is carried out for planar Voronoi percolation). In~\cite{duminil2017exponential, duminil2017sharp, duminil2018poisson} the OSSS inequality is instead applied to \textit{one-arm events directly}, a powerful approach that ultimately yields the phase transition \textit{in all dimensions} for a wide class of models. In our setting there are obstacles to applying the OSSS inequality to the one-arm events directly (explained in more detail in Remark~\ref{r.3d}), but if these could be overcome one might hope to be able to study also the phase transition for smooth Gaussian fields in dimensions $d \ge 3$.

\medskip
Let us now explain how we apply the OSSS inequality in our setting. Since the OSSS inequality as stated above applies only to product measures, our strategy is to exploit the discrete approximation of the white-noise representation introduced in the preceding section. More precisely, for suitably chosen $r, \eps > 0$ that depend on a scale $R > 0$, we study crossing events on the scale~$R$ for the truncated and discretised field 
\[
f^\eps_{r}=q_{r} \star W^\eps 
\]
 by applying the OSSS inequality to the independent Gaussian variables $(\eta_v)_{v \in \eps \Z^2}$ that define the discrete white-noise $W^\eps$. We then compare the truncated and discretised field $f^\eps_{r}$ to our original field $f$ via a `sprinkling' procedure. 

\medskip
The next task is to link the OSSS inequality to a Russo-type formula. As explained in the preceding section, when restricted to any compact domain $f^\eps_{r}$ is finite-dimension, and hence crossing events are measurable with respect to the finite-dimensional i.i.d.\ Gaussian vector of relevant white-noise coordinates $\eta_v$. In this setting there is a simple Russo-type formula (see~\eqref{e.russo_for_iid_Gaussian}) except, just as in the Boolean setting, the `influences' that arise \textit{are not the same as those which appear in the OSSS formula}. Indeed, in the Gaussian setting these influences are only comparable for events $A$ that are \textit{increasing with respect to the variables} $\eta_v$. Since in general this is only true for crossing events if $q \geq 0$, this requires us to impose the strong positivity condition (and is the only place this condition is used).

\medskip
We end this discussion with an observation about a possible alternate, and in some sense more natural, way to apply the OSSS inequality. In~\cite{duminil2017sharp}, the authors extend the OSSS inequality, in the Boolean setting, to \textit{monotonic measures} (here the correct notion of influences are \textit{covariances} between coordinates and events, see~\cite{graham2006influence} where similar influences arise). This gives hope that the OSSS inequality could be applied directly to the field $f$, rather than to the white-noise coordinates $\eta_v$. 

\medskip
To do so, one would first discretise the field by restricting it to a lattice (as in~\cite{bg_16} for instance), and view crossings of quads as paths on this lattice. One might then hope to apply the non-product OSSS inequality directly to the (finite) family of Gaussian variables $f(v)$ that determine each crossing event. However, there are at least two sources of difficulty in realising this approach: 
\begin{itemize}
\item It is not clear that the resulting measures are monotonic in the appropriate sense under the assumption $q \ge 0$; indeed we suspect that `monotonicity' requires the Gaussian field to be \textit{totally positive} (also known as `$MTP_2$'), which is true in the case of the discrete GFF \cite{rodriguez20150}, but is a much stronger assumption than the condition $q \ge 0$.
\item The influences that appear in Russo's formula seem hard to compare to the covariances that appear in the monotonic OSSS formula. 
\end{itemize}

\subsection{Outline of the proof}\label{s.outline}

The overall structure of our proof is similar to that used in previous work \cite{rv_bf}, and consists of three main steps:
\begin{enumerate}
\item (Quasi-independence) Show that crossing events on domains of scale $R$ separated by a distance $R$ are asymptotically independent as $R \to \infty$;
\item (RSW estimates) Apply Tassion's general argument \cite{tassion2014crossing} to deduce the RSW estimates at the zero level $\ell = 0$;
\item (Sharpness) Combine quasi-independence, the RSW estimates, and the OSSS inequality to deduce the existence of the sharp phase transition.
\end{enumerate}

However, we emphasise that our techniques in steps (1) and (3) are completely different from previous approaches, and only the argument in step (2) is unchanged (and borrowed directly from \cite{tassion2014crossing}).

\medskip
 To show quasi-independence, we start from the fact that, for the truncated field $f_R$, the crossing events on domains separated by distance $R$ are genuinely independent. Then we combine our control of the truncation difference $f - f_R$ with an argument based on the Cameron--Martin theorem that bounds the effect of small perturbations on monotonic events, to deduce the approximate independence of these events for $f$. This technique, inspired by the notion of influences (see Section~\ref{s.vl}), ends up being much simpler than previous approaches to showing quasi-independence based on lattice discretisation. We remark that our proof of quasi-independence requires only that $\rho(0) > 0$ (implied by $\kappa \ge 0$) and that Assumption~\ref{a.dec} holds for a given $\beta > 2$, with $\beta$ controlling how quickly the correlations between crossing events converge to zero. This part of the argument also generalises immediately to higher dimensions.
 
\medskip
 To establish the sharp phase transition we use a three-step procedure. As mentioned above, we consider the truncated discretised field $f^{\eps}_{r}$ (for well-chosen $r = r(R)$ and $\eps = \eps(R)$ that are, respectively, growing and decaying  polynomially in~$R$) and apply the OSSS inequality to the product space induced by the discrete white-noise variables $\eta_v$. The first step is then to have RSW and one-arm event estimates for~$f^{\eps}_{r}$. While this could probably be done by suitably modifying the argument from~\cite{tassion2014crossing}, since the law of $f^\eps_{r}$ is only translation invariant on a lattice (and since we want to have estimate uniform in $R$), we instead use a sprinkling procedure to deduce these estimates for $f^{\eps}_{r}$ for small levels $\ell  = \ell(R) > 0$ from the analogous RSW estimates for~$f$. The second step is the application of the OSSS inequality to $f^{\eps}_{r}$, which is done by revealing the white-noise coordinates~$\eta_v$ one-by-one following a classical algorithm that determines a crossing event (see, e.g.,~\cite{schramm2010quantitative} in the case of Bernoulli percolation). This yields a `qualitative' description of the phase transition for a fixed rectangle at large scales. The third and final step is rather standard, and consists in bootstrapping the initial `qualitative' description of the phase transition to complete the proof of the main results. However, when the covariance of the field decays less than exponentially fast, the standard bootstrap (i.e.\ as applies to Bernoulli percolation or the Bargmann-Fock field) no longer works. We overcome this difficulty by making use of a `sprinkled' version of quasi-independence.


\bigskip
\section{Gaussian techniques}
\label{s.gauss}

In this section we collect the arguments that rely on the Gaussian setting of our work. This includes our use of the Cameron--Martin theorem to control the effect of varying the level, and our use of the white-noise representation \eqref{e.wnr} to compare $f$ to truncated and discretised versions (see Section~\ref{s.overview} above). 

\medskip
We assume throughout this section (and the remainder of the paper) that Assumption \ref{a.reg} holds; whenever we need Assumptions~\ref{a.sym}--\ref{a.dec} in addition to this we will make this explicit. 

\subsection{Notation for sub-$\sigma$-algebras}
We begin by introducing notation for certain sub-$\sigma$-algebras that we use in the remainder of the paper:
\begin{defi}\label{d.sigma_al}
We write $\mathcal{F}$ for the classical $\sigma$-algebra on the set $C(\mathbb{R}^2)$ of continuous functions from $\R^2$ to $\R$, i.e.\ the $\sigma$-algebra generated by finite-dimensional projections. For any Borel set $D \subseteq \R^2$, we define the following sub-$\sigma$-algebra $\mathcal{F}_D$ of $\mathcal{F}$ generated by finite-dimensional projections on $D$ (i.e. by $u \mapsto u(x)$ for every $x \in D$). We will use several times the following important fact: if $A \in \mathcal{F}_D$ and $A$ is increasing, then for any $u \in A$ and any continuous function $v$ that satisfies $v_{|D} \geq 0$, we have $u+v \in A$. Finally, for $R > 0$ we abbreviate $\mathcal{F}_R = \mathcal{F}_{B_R}$.
\end{defi}

\subsection{Consequences of the assumptions}
\label{s.cond}

We next collect some important consequences of Assumptions \ref{a.reg} and \ref{a.swp} for the fields $f$ and $f^\eps$. Note that, for each $r \ge 1$, either the function $q_r$ defined in Section \ref{s.td} is identically equal to $0$ or it satisfies Assumptions \ref{a.reg}--\ref{a.dec} whenever $q$ does. As a result, either $f_r$ is identically equal to $0$ or these consequences apply equally to the field $f_r$.

\medskip
We begin by stating some standard facts about Gaussian fields and their derivatives:
\begin{lem}\label{l.prelim_reg}
\
\begin{enumerate}
\item Let $g$ be a centred, almost surely continuous, planar Gaussian field with covariance\footnote{The notation $K \in C^{m,m}$ means that all partial derivatives of $K$ which include at most $m$ differentiations in the first variable and $m$ differentiations in the second variable exist and are continuous.} $K\in C^{k+1,k+1}(\R^2\times\R^2)$. Then, $g$ is almost surely $C^k$, $(g,\partial^{\alpha}g)_{\alpha \, : \, |\alpha| \leq k}$ is a centred Gaussian field, and for all multi-indices $\alpha_1$, $\alpha_2$ such that $|\alpha_1| \leq k$ and $|\alpha_2| \leq k$ and for every $x,y \in \R^2$, we have
\[
\E \left[ \partial^{\alpha_1} g(x) \partial^{\alpha_2} g(y) \right] =  \partial^{\alpha_1}_x\partial^{\alpha_2}_y K(x,y)  .
\]
\item If $g$ is a centred planar Gaussian field with covariance $K\in C^{1,1}(\R^2\times\R^2)$, then there exists a modification of $g$ which is continuous.
\item Let $g$ be a centred, almost surely continuous, planar Gaussian field with covariance $K \in C^{2,2}(\R^2,\R^2)$, let $h$ be a centred planar $L^2$ field, and define $\widetilde{K}(x,y)=\E \left[ g(x) h(y) \right]$. Then, for every multi-index $\alpha$ such that $|\alpha| \leq 1$ and for every $x,y \in \R^2$, $\partial_x^\alpha \widetilde{K}(x,y)$ exists and satisfies
\[
\E \left[ \partial^\alpha g(x) \, h(y) \right] = \partial_x^\alpha \widetilde{K}(x,y)  .
\]
\end{enumerate}
\end{lem}
\begin{proof}
The first and second statements can be found in Appendices A.3 and A.9 of~\cite{nazarov2015asymptotic}. For the last statement, note that the first statement implies that $g$ is $C^1$. Moreover, we have:
\begin{align*}
\E \left[ \partial^{(1,0)} g(x) \, h(y) \right] & =  \E \left[ \lim_{a \rightarrow 0} \frac{g(x_1+a,x_2)-g(x)}{a} h(y) \right]\\
&=  \lim_{a \rightarrow 0} \E \left[ \frac{g(x_1+a,x_2)-g(x)}{a} h(y) \right]\\
& =  \lim_{a \rightarrow 0} \frac{\widetilde{K}((x_1+a,x_2),y)-\widetilde{K}(x,y)}{a}  .
\end{align*}
The second inequality comes from the fact that: i) the almost sure convergence of the Gaussian variables $\frac{g(x_1+a,x_2)-g(x)}{a}$ is equivalent to the $L^2$ convergence, and ii) $h(y)$ is $L^2$. This completes the proof in the case $\alpha=(1,0)$, with the other case identical.
\end{proof}

We next use Lemma~\ref{l.prelim_reg} to deduce some consequences of Assumption \ref{a.reg}. Let us stress that, thanks to this assumption, we can use dominated convergence in order to exchange derivatives and convolution of $\kappa=q \star q$ for derivatives of order at most $3$, a fact we use below without mentioning it explicitly.

\medskip
For each quad $Q = (D; \gamma, \gamma')$, let $Q^\ast$ denote the quad with domain $D$ and boundary arcs $\nu, \nu'$ defined so that $\nu \cup \nu' =  \partial D \setminus \overline{\gamma \cup \gamma'}$. For instance, if $D$ is a rectangle and $\gamma$ and $\gamma'$ are the `left' and `right' edges, then $\nu$ and $\nu'$ consist of the `bottom' and `top' edges. Let $\cross_{\ell}^\ast(Q)$ denote the crossing event for the quad $Q^\ast$ and the set $\mathcal{E}_\ell^c$, i.e.\ the event that there is a connected component of $\mathcal{E}_\ell^c$ whose intersection with $Q^\ast$ intersects both $\nu$ and $\nu'$.

\begin{proposition}
\label{p.areg}
\
\begin{enumerate}
\item Fix $\eps> 0$. Then there exist continuous modifications of $f = q \star W$ and $f^\eps = q \star W^\eps$. Moreover, $f$ and $f^\eps$ are almost surely $C^2$ and, for every multi-index $\alpha$ such that $|\alpha| \le 1$,
\begin{equation}
\label{e.intconder}
  \partial^\alpha f = (\partial^\alpha q) \star W  \quad \text{and} \quad   \partial^\alpha f^\eps = (\partial^\alpha q) \star W^\eps .
  \end{equation}
\item The field $f$ is non-degenerate, meaning that, for each $k \in \mathbb{N}$ and distinct points $x_1, \ldots , x_k \in \mathbb{R}^2$, $(f(x_1), \ldots , f(x_k))$ is a non-degenerate Gaussian vector. 
\item For each $\ell \in \mathbb{R}$, the level set $\mathcal{L}_\ell$ almost surely consists of a collection of simple curves. Moreover, for each $\ell \in \mathbb{R}$ and quad $Q$, almost surely
\[
  \{ f \in \cross_{\ell}(Q)  \} = \{  f \notin \cross_\ell^\ast(Q)  \} . \] 
\end{enumerate}
\end{proposition}

\begin{proof}
\textbf{(1).} We first consider $f = q \star W$. Since $q \in L^2$, and by the definition of the white-noise~$W$, $q \star W$ is a stationary planar Gaussian field with covariance $\kappa = q \star q$. Since Assumption~\ref{a.reg} implies that $\kappa$ is $C^6$ (i.e.\ the covariance of $f$ is $C^{3,3}$), the first two statements of Lemma \ref{l.prelim_reg} guarantee the existence of a continuous modification of $f$ that is almost surely $C^2$. Concerning~\eqref{e.intconder}, Assumption~\ref{a.reg} ensures that $\partial^\alpha q \in L^2$, and so $(\partial^\alpha q) \star W$ is well-defined. Then we use dominated convergence (to exchange derivatives and convolution), the definition of $W$, and the first and third statements of Lemma~\ref{l.prelim_reg} to verify that $\E \left[ (\partial^\alpha f - (\partial^\alpha q) \star W )^2 \right] = 0$. To be more precise, by the first statement of Lemma~\ref{l.prelim_reg} and by the definition $W$,
\begin{align*}
& \E \left[ \big(\partial^\alpha f(x) -  ((\partial^\alpha q) \star W)(x)  \big)^2 \right]\\
& = \partial^{2\alpha} \kappa(0) -2 \E \left[ \partial^\alpha f(x) \, ((\partial^\alpha q) \star W) (x) \right] + (\partial^\alpha q )\star (\partial^\alpha q) (0)\\
& = (\partial^\alpha q )\star (\partial^\alpha q) (0) -2 \E \left[ \partial^\alpha f(x) \, ((\partial^\alpha q) \star W) (x) \right] + (\partial^\alpha q )\star (\partial^\alpha q) (0) .
\end{align*}
Moreover, by the third statement of Lemma~\ref{l.prelim_reg},
\[
\E \left[ \partial^\alpha f(x) \,  (\partial^\alpha q) \star W (x) \right] = \partial^\alpha_x \left( q \star (\partial^\alpha q) (x-y) \right)_{|x=y} = (\partial^\alpha q )\star (\partial^\alpha q) (0)  .
\]
As a result, $\E \left[ (\partial^\alpha f(x) - (\partial^\alpha q) \star W (x) )^2 \right]=0$ for every $x$. Hence, by considering a continuous modification of $(\partial^\alpha q) \star W$, we have \eqref{e.intconder}.

\medskip
Turning to $f^\eps = q \star W^\eps$, Assumption~\ref{a.reg} ensures that both~$q$ and $\partial^\alpha q$, restricted to the lattice~$\eps \mathbb{Z}^2$, are square-summable, which ensures that $f^\eps = q \star W^\eps$ and $\partial^\alpha q \star W^\eps$ are well defined stationary Gaussian fields. The remainder of the proof is essentially the same as in the first case.

\medskip
\textbf{(2).} The non-degeneracy of $f$ follows from the fact that the support of $\rho$ (and hence also the support of the spectral measure $\mu$) contains an open set \cite[Theorem 6.8]{wendland05}.

\medskip
\textbf{(3).} For this, we refer to \cite{adta_rfg} or to Lemma~A.9 of \cite{rv_rsw}.
\end{proof}

To finish, we state some consequences of Assumption \ref{a.swp}.

\begin{proposition}
\label{p.aswp}
The strong positivity condition in Assumption \ref{a.swp} implies the weak positivity condition $\kappa \ge 0$. In turn, the weak positivity condition is equivalent to the \textit{FKG inequality} for finite-dimensional projections of the field $f$, i.e., for every $k \in \mathbb{N}$, every $\{x_1, \ldots, x_k\} \subset \mathbb{R}^2$ and every increasing Borel sets $A,B \subseteq \R^n$:
\begin{equation} \label{e.fkg}
  \mathbb{P}[f|_{x_1, \ldots, x_k}  \in A \cap B] \ge \mathbb{P}[f|_{x_1, \ldots, x_k}  \in A] \, \mathbb{P}[f|_{x_1, \ldots, x_k}  \in B]  .
  \end{equation}
Moreover, the weak positivity condition implies that $\rho(0) > 0$, and if we assume furthermore that Assumption~\ref{a.dec} holds for some $\beta > 2$, then there exists a neighbourhood $V$ of $0$ such that $\inf_V|\rho|  > 0$.
\end{proposition}
\begin{proof}
Clearly $q \ge 0$ implies that $\kappa = q \star q \ge 0$. Moreover, \eqref{e.fkg} is well-known to be equivalent to $\kappa \ge 0$ by a result of Pitt \cite{pitt1982positively}. Finally, since $\rho^2 = \mathcal{F}[\kappa]$ and $\kappa$ is non-negative and not identically zero, $\rho^2(0) > 0$. If we assume furthermore that $\beta > 2$, then $\rho$ is continuous (by dominated convergence), which gives the last property.
\end{proof}

\begin{remark}
By approximation arguments (see, e.g., Appendix A.2 of \cite{rv_rsw}), \eqref{e.fkg} implies the positive association of each of the compactly-supported increasing events that we consider in this paper, i.e.\ for every compact $D \subset \mathbb{R}^2$ and every increasing $A, B \in \mathcal{F}_D$ that is mentioned in this paper,
\begin{equation} \label{e.fkg2}
 \mathbb{P}[f \in A \cap B] \ge \mathbb{P}[f \in A] \mathbb{P}[f  \in B] . 
 \end{equation}
While we actually believe that \eqref{e.fkg} implies \eqref{e.fkg2} for \textit{all} compactly supported increasing events, we are unaware of any such statement in the literature.
\end{remark}

\subsection{Varying the level via the Cameron--Martin theorem}
\label{s.vl}

In this section we show how to control the effect of varying the level on monotonic events, in particular giving a bound in terms of the spectral density $\rho$ contained in a small annulus centred at the origin. For each $0 < r_1 < r_2$, let $\ann_{r_1, r_2}$ denote the Euclidean annulus centred at the origin with radii $r_1$ and~$r_2$.

\begin{prop}\label{p.cameron_martin}
There exist absolute constants $c_1, c_2  > 0$ such that, for each $R > 0$, monotonic event $A \in \mathcal{F}_R$, and $t \in \R$,
\[ \left| \Pro \left[ f \in A   ]  - \mathbb{P}[ f - t \in A \right] \right| \leq  \frac{c_1  R \, |t|  }{      \inf\{ |\rho(x)| : x \in  \ann_{c_2/R, 2c_2/R} \} } .\]
In particular, if $\gamma \ge 0$ is such that $\rho(x) > c_3 |x|^{\gamma}$ for a constant $c_3 > 0$ and sufficiently small $|x|$, then there exist $c, R_0 > 0$ such that, for every $R > R_0$, monotonic event $A \in \mathcal{F}_R$, and $t \in \R$,
\[ \left|  \Pro \left[  f \in A  ]  - \mathbb{P}[ f - t \in A  \right] \right|  \le  c R^{1+\gamma} |t|  .\]
\end{prop}
Let us note that, if we follow the proof of Proposition~\ref{p.cameron_martin} and if we assume that $|t| \leq M$ for some $M>0$, then we even have
\[ \left| \Pro \left[ f \in A   ]  - \mathbb{P}[ f - t \in A \right] \right| \leq  \frac{c_1(M)  R \, |t|  \sqrt{\Pro \left[ f \in A \right]}  }{      \inf\{ |\rho(x)| : x \in  \ann_{c_2/R, 2c_2/R} \} } \]
for some $c_1(M)>0$. For simplicity, and since we do not need this sharper estimate, we have stated the above for every $t$ and without the term $\sqrt{\Pro \left[ f \in A \right]}$. However, we believe that this sharper estimate could be useful to study events $A$ with very small probability.

\medskip
Proposition \ref{p.cameron_martin} has an easy corollary when the infimum of $|\rho|$ on some neighbourhood of $0$ is positive, which is what we apply in the sequel.

\begin{corollary}\label{c.cameron_martin}
Suppose that there exists a neighbourhood $V$ of $0$ such that $\inf_V|\rho| > 0$. Then, there exist $c, R_0 > 0$ such that, for every $R > R_0$, monotonic event $A \in \mathcal{F}_R$, constants $\eps, \delta \ge 0$, and continuous planar random field $g$ such that $\Pro \left[ \| f-g \|_{\infty,B_R} \geq \eps \right] \leq \delta$,
\[ \Pro \left[ \{ f \in A \} \bigtriangleup \{ g \in A \} \right] \leq c R \eps  + \delta  .\]
\end{corollary}

\begin{remark}\label{r.bernou}
Proposition \ref{p.cameron_martin} can be viewed as the analogue, in the Gaussian setting, of well-known inequalities that hold for Boolean functions. Let $n \in \N$ and consider the product space $\Omega_n = \{ -1,1 \}^n$ equipped with the product probability measure $\Pro^n_p = (p \delta_1 + (1-p) \delta_{-1})^{\otimes n}$. Then for every $\eps > 0$ there is a $c > 0$ such that for every $A \subseteq \Omega_n$ and  $\eps \leq p \leq q \leq 1-\eps$,
\begin{equation}\label{e.cmber}
\left| \Pro_p \left[ A \right] - \Pro_q \left[ A \right] \right| \leq c |p-q| \sqrt{n} . 
 \end{equation}
The proof of \eqref{e.cmber} goes as follows. Define the functions
\[
\chi_i^p \, : \, \omega \in \Omega_n \mapsto \sqrt{\frac{1-p}{p}} \un_{\omega_i=1} - \sqrt{\frac{p}{1-p}} \un_{\omega_i=-1} ,
\]
which from an orthonormal set of centred variables of the $L^2$ space $L^2(\Omega_n,\Pro^n_p)$. Applying a differential formula, whose proof is similar to the classical Russo's formula \eqref{e.rus1}, we obtain
\[
\frac{d}{dp}\Pro_p^n \left[ A \right] = \frac{1}{2\sqrt{p(1-p)}} \sum_{i=1}^n \E_p^n \left[ \chi^p_i(\omega) \un_A \right] .
\]
By the Cauchy--Schwarz inequality, we deduce that
\begin{equation}\label{e.CS_for_bool}
\left| \sum_{i=1}^n \E_p^n \left[ \chi^p_i(\omega) \un_A \right] \right| \leq  \sqrt{n}  \left( \sum_{i=1}^n \E_p^n \left[ \chi_i^p(\omega) \un_A \right]^2 \right)^{1/2} \, .
\end{equation}
Since the functions $\chi_i^p$ form an orthonormal set, Parseval's formula gives that 
\[ \sum_{i=1}^n \E_p^n \left[ \chi_i^p(\omega) \un_A \right]^2  \leq \E_p^n \left[ \un_A^2 \right] \leq 1 , \]
 which ends the proof. Note that, when $A$ is increasing, $\E^n_p \left[ \chi^p_i \un_A \right]$ is a constant (that depends on~$p$) times the influence of $A$ (see Section~\ref{s.OSSSsub} for the notion of influence of Boolean events). Hence, one interpretation of~\eqref{e.CS_for_bool} is that the total influence (i.e.\ the sum of all influences) is of order at most $\sqrt{n}$; this idea guides us also in establishing Proposition \ref{p.cameron_martin}. One can also find this idea in the work of Benjamini--Schramm~\cite{benjamini1998conformal} on the conformal invariance of Voronoi percolation, in which they use an analogous control of the influences in order to bound the `number of defects' (see their Section 8.1).

\medskip
With a very similar proof (more precisely, by noting that, if $X_1,\cdots,X_n$ are i.i.d.\ standard Gaussian variables with mean $\ell$, then the random variables $(X_1-\ell,\cdots,X_n-\ell)$ form an orthonormal set of the underlying $L^2$ space), we obtain the following. Let $\nu_\ell^n$ denote the law of~$n$ i.i.d.\ Gaussian variables of mean $\ell$ and let $A$ be a Borel subset of $\R^n$. Then, for each $\ell_1, \ell_2 \in \mathbb{R}$,
\[ \left| \nu_{\ell_1}^n(A) - \nu_{\ell_2}^n(A) \right| \leq \sqrt{n} |\ell_1-\ell_2|  . \]
Proposition \ref{p.cameron_martin} can be viewed as a generalisation of this statement to continuous Gaussian fields, taking $\text{area}(B_R)$ as the analogue of the $n$. \end{remark}

\medskip
To prove Proposition \ref{p.cameron_martin} we shall need to introduce some of the standard theory of Gaussian fields, and in particular the Cameron--Martin theorem; for this we refer to \cite[Chapters VIII and XIV]{jan_97}.

\medskip
Recall that to our Gaussian field $f$ we can associate a Hilbert space of functions $H \subset C(\mathbb{R}^2)$ known as the \textit{Cameron--Martin space}, defined in the following way. First, let $G$ denote the Hilbert space of centred Gaussian random variables that is the closure in $L^2$ of the linear span of $\{f(x)\}_{x \in \mathbb{R}^2}$, i.e.\ the set
\[   \sum_{i \in \mathbb{N}} a_i f(x_i)  , \quad    x_i \in \mathbb{R}^2 , a_i \in \mathbb{R}, \sum_i a_i a_j K(x_i,x_j) < \infty .\]
Then define the (injective) linear map $P: G \to  C(\mathbb{R}^2)$ by
\[    \xi \mapsto P(\xi)(\cdot) := \langle \xi, f(\cdot) \rangle_G = \mathbb{E}[\xi f(\cdot) ] . \]
The function space $H = P(G) \subset C(\mathbb{R}^2)$, equipped with the inner product
\[    \langle h_1, h_2 \rangle_H =  \langle P^{-1}(h_1), P^{-1}(h_2) \rangle_G = \mathbb{E}[P^{-1}(h_1) P^{-1}(h_2) ] ,\]
 is a Hilbert space known as the \textit{Cameron--Martin space}; by construction, $P$ defines an isometry between $G$ and $H$.

\medskip
An equivalent description of $H$ is as the completion of the space of finite linear combinations of the covariance kernel $K$
\[    \sum_{1 \le i \le n} a_i K( s_i , \cdot)   \ , \quad a_i \in \mathbb{R}, s_i \in \mathbb{R}^2 , \]
 equipped with the inner product  
\begin{equation}
\label{e:rkhsip}
  \bigg\langle \sum_{1 \le i \le n} a_i K(s_i, \cdot), \sum_{1 \le i \le n} a'_i K(s'_i, \cdot) \, \bigg\rangle_H = \sum_{1 \le i,j \le n}   a_i a_j'  K(s_i, s_j') \, .
  \end{equation}
  This latter construction emphases the role of the covariance kernel $K$ in the construction of $H$; indeed one sees that $H$ is the (unique) \textit{reproducing kernel Hilbert space} associated to $K$, i.e., for any $h \in H$ and $x \in \mathbb{R}^2$,  
  \begin{equation}
  \label{e.rkprop}
    \langle h(\cdot), K(x, \cdot) \rangle_{H} = h(x) . 
    \end{equation}
 
 \medskip
 Let us now state the well-known \textit{Cameron--Martin theorem}, which describes how the law of $f$ changes under translation by an element of $H$. 
 
 \begin{theorem}[Cameron--Martin; see {\cite[Theorems 14.1 and 3.33]{jan_97}}]
 \label{t.cameron_martin}
 Suppose $h \in H$. Then the law of $f+h$ equals the law of $f$ with Radon--Nikodym derivative
 \[    \exp \left\{ P^{-1}(h) - \frac{1}{2} \mathbb{E}[P^{-1}(h)^2]  \right\} . \]
 In particular, for each $A \in \mathcal{F}$,
 \[    \mathbb{P}[ f + h \in A ]  = \mathbb{E} \left[   \exp \left\{ P^{-1}(h) - \frac{1}{2} \mathbb{E}[P^{-1}(h)^2]  \right\}  \id_{f \in A}  \right]  .\]
 \end{theorem} 

Note that the map $P^{-1}(\cdot)$ plays the central role in the Cameron--Martin theorem; this is often called the \textit{Paley--Weiner map} and one of its key properties is that, since $P$ is an isometry, its image $P^{-1}(h)$ is a random variable with distribution $\mathcal{N}(0, \|h\|^2_H)$. 

\medskip
We next state a corollary of the Cameron--Martin theorem that is all that we shall need to prove Proposition \ref{p.cameron_martin}; since we were unable to find this statement in the literature, we give a short proof. 

\begin{corollary} \label{l.cameron_martin}
For every $h \in H$ and $A \in \mathcal{F}$:
\[ \left| \mathbb{P}[ f  \in A  ] - \mathbb{P}[f + h \in A ]  \right| \le   \frac{\| h \|_{H} }{\sqrt{\log 2}}   . \]
\end{corollary} 
\begin{proof}
Abbreviate $X = P^{-1}(h)$. By Theorem \ref{t.cameron_martin}, and the linearity of $\mathbb{E}$,
\begin{equation}
\label{e:cm1}
   \mathbb{P}[ f  \in A  ] - \mathbb{P}[f +  h \in A ]    =  \mathbb{E} \left[    \left( 1 - \exp \left\{    X - \frac{1}{2} \mathbb{E}[X^2]   \right\} \right)  \id_{f \in A}  \right]  , 
   \end{equation}
which, by the Cauchy--Schwarz inequality, is in absolute value at most
\[      \bigg(\mathbb{E} \Big[    \Big( 1 - \exp \Big\{  X - \frac{1}{2} \mathbb{E}[X^2]  \Big\}  \Big)^2 \Big] \Pro \left[ f \in A \right] \bigg)^{1/2}  ; \]
this bound is analogous those holding in the context of Boolean functions (see Remark~\ref{r.bernou} above). Since $X$ is distributed as $\mathcal{N}(0, \|h\|^2_H)$, and by standard properties of the normal distribution,
   \begin{align*}
    \mathbb{E} \Big[    \Big( 1 - \exp \Big\{  X - \frac{1}{2} \mathbb{E}[X^2]   \Big\}  \Big)^2 \Big] & =      \mathbb{E} \left[     1 - 2 e^{   X - \frac{1}{2} \mathbb{E}[X^2]  }  +  e^{   2  X -  \mathbb{E}[X^2]  }  \right]  
    \\  & = 1 - 2 + e^{  \mathbb{E}[X^2]}  =  e^{  \|h\|^2_{H}} - 1  ,
    \end{align*}
  and hence we have shown that
    \[    \left| \mathbb{P}[ f  \in A  ] - \mathbb{P}[f + h \in A ]  \right| \le \min \left\{   \sqrt{e^{  \|h\|^2_{H}} - 1} , 1 \right\} .  \]
 To conclude we use the fact that
 \[   \min \{   \sqrt{e^{x^2} - 1}  , 1 \} \le x / \sqrt{ \log 2 }  \]
for all $x \ge 0$, which is simple to verify (by the convexity of  $\sqrt{e^{x^2} - 1}$ for instance).
\end{proof}

To complete the proof of Proposition \ref{p.cameron_martin}, we need one additional element of the theory of Cameron--Martin spaces that is specific to the setting of stationary Gaussian fields (see, e.g., \cite[Eq.(2.4), p.67]{bertho04} or \cite[Lemma 3.1]{kimwah70}). Recall the spectral density $\rho^2$, and let $S = \text{supp}(\rho)$. An alternative description of the Cameron--Martin space $H$ is as the set 
 \begin{equation}
 \label{e.altcms}
  \bar{H} =  \mathcal{F}[ g \rho ] \ , \quad g \in L^2_\text{sym}(S), 
 \end{equation}
  equipped with the inner product inherited from $L^2_\text{sym}$, where $L^2_\text{sym}(S)$ denotes the set of complex Hermitian $L^2$ functions supported on $S$. To see why this is true, observe that the map $g \mapsto \mathcal{F}[ g \rho ] $ defines an isometry between $L^2_\text{sym}(S)$ and $\bar{H}$, and so the latter is a Hilbert space. By the uniqueness of the RKHS \cite[Theorem F.7]{jan_97}, it remains to verify the reproducing kernel property \eqref{e.rkprop}, i.e.\ for every $F[ g \rho] \in \bar{H}$ and $x \in \mathbb{R}^2$ we verify that
  \[ \langle \mathcal{F}[ g \rho] (\cdot) , \kappa(\cdot - x)  \rangle_{\bar{H}} =  \langle \mathcal{F}[ g \rho] (\cdot) , \mathcal{F}[ \rho^2  ] (\cdot - x)  \rangle_{\bar{H}} =  \langle g , \rho  e^{- 2 \pi i \langle s, x \rangle} \rangle_{L^2_{\text{sym}}}   = \mathcal{F}[ g \rho](x) , \] 
  where in the second equality we used the translation identity for the Fourier transform.
  
  \medskip
  One consequence of the representation in \eqref{e.altcms} is the identity
  \[   \|h\|^2_H = \int_{x \in \mathbb{R}^2} |\hat h(x)|^2 / \rho^2(x) \, dx  \]
  valid for any $h \in H$ such that $\hat h = \mathcal{F}[h]$ is defined. In particular, if $\text{supp}(|\hat h|)$ has finite area we have the bound
    \begin{equation}
  \label{e:rkhsnormbound}
    \| h \|^2_{H} \le  \frac{ \sup \{ |\hat h(x)|^2  : x \in \text{supp}(|\hat h|) \} \, \text{Area}( \text{supp}(|\hat h| ) ) }{ \inf\{   \rho^2(x) : x \in \text{supp}(|\hat h|) \} }. 
    \end{equation}

\begin{proof}[Proof of Proposition \ref{p.cameron_martin}]
Without loss of generality we take $A$ increasing and $t \leq 0$. Since $A$ is increasing and belongs to $\mathcal{F}_R$,
    \[  0 \le  \mathbb{P}[ f  \in A  ] - \mathbb{P}[f+ t \in A ]  \le   \mathbb{P}[ f  \in A  ] - \mathbb{P}[f + t h \in A ]  \] 
    for each $h$ that satisfies $h_{|B_R} \ge 1$. In light of Corollary \ref{l.cameron_martin}, it remains only to show that 
    \[  \inf_{h \in {H} \, : \, h_{|B_R} \ge 1 } \| h \|_{H} \le   \frac{   c_1 R   }{    
    \inf\{ |\rho(x)| : x \in  \ann_{c_2 / R, 2 c_2/ R} \} }     \]
    for suitable $c_1, c_2  > 0$. For this fix $c > 0$ sufficient small so that the Fourier transform of the (normalised) identify function on the annulus $\ann_{c, 2c } $ is larger than $1$ on $B_1$, i.e.\
        \[ \mathcal{F}[  c^{-2} \id_{ \ann_{c, 2c } } ]   \ge 1 \text{ on } B_1 ; \]
    such a $c > 0$ is easily checked to exist. Then define
        \[ h = \mathcal{F}[  R^2  c^{-2}  \id_{ \ann_{c/R, 2c/R } } ]   ,  \]
which by the scaling of the Fourier transform satisfies $h_{|B_R} \ge 1$ for all $R > 0$. By \eqref{e:rkhsnormbound}, 
        \[ \| h \|^2_{H} \le  \frac{ R^2  c^{-2} }{ \inf\{ \rho^2(x)  :  x \in \ann_{c/R, 2c/R }  \}}    , \]
        which completes the proof.
    \end{proof}

\subsection{Comparison to truncated and discretised versions of the field}

We now show how to compare the field $f$ to its truncated and discretised versions $f_r$ and $f^\eps$ introduced in Section~\ref{s.overview}. Our comparison is in terms of the sup-norm, since this is enough for our application to crossing events, but stronger comparisons would be easy to prove using similar methods (strengthening the assumptions as appropriate).

\begin{prop}\label{p.coupling}
Suppose that \eqref{e.beta} holds for $\beta > 1$. Then there exist $c_1, c_2 > 0$ such that, for all $\eps > 0$ and $R, r \ge  1$ and $t \ge \log R$,
\[ \Pro \left[ \| f-f^\eps \|_{\infty,B_R} + \| f_r-f^\eps_r \|_{\infty,B_R}   \geq c_1 t \, \eps \right] \leq c_1  e^{-c_2 t^2 }  . \]
Moreover, there exist $c_1, c_2 > 0$ such that, for all $R, r \ge 1$ and $t \ge \log R$,
\[ \Pro \left[ \| f-f_r \|_{\infty,B_R} \geq c_1 t \, r^{1-\beta}  \right] \leq c_1 e^{-c_2 t^2} .  \]
\end{prop}

The proof of Proposition \ref{p.coupling} will be a straight-forward combination of the following lemmas:

\begin{lem}\label{l.btis}
There exists an absolute constant $c > 0$ such that, for every $C^1$ planar Gaussian field $g$, and for all $R_1 \ge c$ and $R_2  \ge  \log R_1$,
\[ \Pro \left[ \| g \|_{\infty,B_{R_1}} \ge  m R_2 \right] \leq e^{-  R_2^2/c   }  , \]
where 
\begin{equation}
\label{e.m}
m = \Big( \sup_{x \in \mathbb{R}^2} \sup_{|\alpha| \le 1} \mathbb{E}[  (  \partial^\alpha g)^2 (x) ]  \Big)^{1/2}   . 
 \end{equation}
\end{lem}

\begin{lem}\label{l.ube1}
Suppose $G: \R^+ \to \R^+$ is such that, for all $x \in \mathbb{R}^2$ and $|\alpha| \le 1$,
\begin{equation}
\label{e.ube1}
|\partial^\alpha q (x)|  < G(|x|)  .
  \end{equation}
  Suppose also that $ \int_{s > 1}  s G(s)^2 \, ds  < \infty$. Then there exists a $c > 0$ such that, for all $r \ge 1$,
\[  \sup_{x \in \mathbb{R}^2} \sup_{|\alpha| \le 1} \mathbb{E}[   (\partial^\alpha (f - f_r) )^2(x) ]     < c   \int_{s > r}  s G(s)^2 \, ds .  \]
\end{lem}

\begin{lem}\label{l.ube2}
There exists a $c > 0$ such that, for all $\eps > 0$ and $r \ge 1$, 
\[   \sup_{x \in \mathbb{R}^2} \sup_{|\alpha| \le 1} \mathbb{E}[   (\partial^\alpha (f - f^\eps) )^2(x) ]  + \mathbb{E}[   (\partial^\alpha (f_r - f_r^\eps) )^2(x) ]  < c \eps^2   . \]
\end{lem}

Before proving these lemmas, let us complete the proof of Proposition \ref{p.coupling}.

\begin{proof}[Proof of Proposition \ref{p.coupling}]
We prove only the second statement; the first is proven similarly. By assumption there are $c_1, c_2 > 0$ such that \eqref{e.ube1} holds with 
\[   G(x) = c_1 x^{-\beta} \ , \quad \beta > 1 .\]
Evaluating the integral $\int_{s > r} s G(s)^2 \, ds < \infty$, by Lemma \ref{l.ube1} there are $c_3, c_4 > 0$ such that, for $r \ge 1$, 
\[ \sup_{x \in \mathbb{R}^2} \sup_{|\alpha| < 1} \mathbb{E}[   (\partial^\alpha (f - f_r) )^2(x) ]    <  c_3 r^{2(1-\beta)}.  \]
Applying Lemma \ref{l.btis} we have the result as long as $R$ is sufficiently large. We deduce the result for all $R \ge 1$ by replacing $c_1$ with a suitably large constant.
\end{proof}

We now turn to the proof of Lemmas \ref{l.btis}--\ref{l.ube2}. The proof of Lemma \ref{l.btis} is an easy consequence of two classical results: Kolmogorov's theorem \cite[A.9]{nazarov2015asymptotic} and the Borell--TIS inequality (see \cite[Theorem 2.9]{azws} and \cite{adta_rfg}). 

\begin{proof}[Proof of Lemma \ref{l.btis}]
By stationarity and Kolmogorov's theorem, there is an absolute constant $c_1 > 0$ such that
\[   \mathbb{E}[ \| g \|_{\infty,B_1} ] < c_1 m   ,
\]
where $m$ is the constant defined in \eqref{e.m}. Moreover, by definition we have $\sup_{x \in B_1} \mathbb{E}[  g(x)^2 ] \leq m^2$. An application of the Borell--TIS inequality to the field $g|_{B_1}$ yields that, for each $u > 0$,
\[  \mathbb{P}[  \| g \|_{\infty,B_1}   \ge c_1 m  + u ] \le e^{ - u^2 / ( 2c_1^2 m^2)  } . \]
Setting $u = m(R_2 - c_1)$ and tiling $B_{R_1}$ with $\asymp  R_1^2$ disjoint boxes of size $1$ (since we can assume that $R_1 \ge 1$) yields, by stationarity and the union bound, that there exists $c_2 > 0$ such that
\[ \mathbb{P}[  \| g \|_{\infty,B_{R_1}}   \ge m R_2 ] \le c_2 R_1^2 e^{ - (R_2 - c_1)^2/ (2 c_1^2)  }.  \]
Setting $c_3 > 0$ to be sufficiently large such that for all $R_1 > c_2$ and for all $R_2 \ge \log R_1$, 
\[ \frac{(R_2 - c_1)^2}{2 c_1^2} - \log(  c_2 R_1^2  ) >   \frac{R_2^2}{4 c_1^2}  \, ,  \]
we have the result for $c$ the maximum of $4c_1^2$ and $c_2$.
\end{proof}

\begin{proof}[Proof of Lemma \ref{l.ube1}]
By stationarity, it is sufficient to prove the result for $x = 0$. Recall from Section \ref{s.overview} and Proposition \ref{p.areg} that $f - f_r = (q-q_r) \star W$, where $q_r = q \chi_r$, and also that  
\begin{equation}
\label{e.wn1}
 \partial^\alpha (f - f_r) =  (\partial^\alpha (q-q_r) ) \star W =  \int (\partial^\alpha (q-q_r)) (\cdot - u) \, dW(u) .
 \end{equation}
By standard properties of white-noise, the second moment of the integral in \eqref{e.wn1} is
\[  \int   (\partial^\alpha (q-q_r))^2 (\cdot - u)  \, du  .\]
Since $G$ is defined to satisfy, for all $x \in \mathbb{R}^2$,
\[ \sup_{|\alpha| \le 1 }  |\partial^\alpha q (x)|  < G(|x|)  , \]
and since the function $(1 - \chi_r)$ and all its derivatives are equal to zero on $B_r$ and uniformly bound elsewhere, this integral is at most
\[   c   \int_{|u| > r} G(|u|)^2 du \]
for a certain constant $c > 0$. After switching to polar coordinates, we have the result.
\end{proof}

Finally, for the proof of the Lemma \ref{l.ube2} we shall need the following auxiliary lemma:

\begin{lemma}[Piece-wise constant approximation]
\label{l.pc}
Let $g: \mathbb{R}^2 \to \mathbb{R}$ be a $C^1$ function that is also in the Sobolev space $H^1$ (i.e.\ for every multi-index~$\alpha$ such that $|\alpha| \le 1$, $\partial^\alpha g \in L^2$). Recall, for each $x \in \mathbb{R}^2$ and $\varepsilon > 0$, the piece-wise constant approximation $g^{x, \eps}$ defined in~\eqref{e.pwc}. Then there exists a constant $c > 0$, depending only on $\| g \|_{H^1}$, such that, for all $\eps > 0 $, 
\[   \sup_{x \in \mathbb{R}^2}  \| g - g^{x, \eps} \|_2^2 < c \eps^2  .  \]
\end{lemma}
\begin{proof}
Fix $x \in \mathbb{R}^2$ and $\varepsilon > 0$. For each $v \in \eps \Z^2$ define the lattice box $D_v = x + v + [-\eps/2, \eps/2]^2$. Since $g^{x, \eps}$ is the mean of $g$ on $D_v$, the Poincar\'{e}--Wirtinger inequality implies the existence of an absolute constant $c' > 0$ such that
 \begin{equation}
 \label{e.pw}
  \int_{u \in D_v} (g - g^{x, \eps})^2  \, du \le  c' \eps^2   \int_{u \in D_v} |\nabla g|^2 \, du . 
  \end{equation}
Summing \eqref{e.pw} over $v \in \eps Z^2$ yields the result, since $g \in H^1$.
\end{proof}

\begin{proof}[Proof of Lemma \ref{l.ube2}]
Recall from Section \ref{s.overview} that we can express 
\[   (  f -  f^\eps )(x) = \int \left(  q - q^{x, \eps} \right) (x- u) \, dW(u)  . \]
Moreover, by Proposition \ref{p.areg} we have
\begin{eqnarray*}
( \partial^\alpha f - \partial^\alpha f^\eps )(x) & = & \int \partial^\alpha q(x-u) dW(u) - \int \partial^\alpha q(x-u) dW^\eps(u)\\
& = & \int \left(   (\partial^\alpha q)  - (\partial^\alpha q)^{x, \eps}  \right) (x- u) \, dW(u) ,
\end{eqnarray*}
where $(\partial^\alpha q)^{x, \eps}$ denotes the piece-wise constant approximation of the function $\partial^\alpha q$ defined in~\eqref{e.pwc}. The second moment of this quantity is equal to 
\[  \| \partial^\alpha q - (\partial^\alpha q)^{x, \eps} \|_2^2 . \]
Since $\partial^\alpha q \in H^1$, applying Lemma \ref{l.pc} yields the result. Moreover, since the constant in Lemma~\ref{l.pc} depends only on $\| \cdot \|_{H^1}$, and since $\|\partial^\alpha q_r\|_{H^1}$ is uniformly bounded over $r \ge 1$ and $|\alpha| \leq 1$ by the assumptions on $\chi_r$, the same proof works also for $f_r - f_r^\eps$.
\end{proof}


\bigskip

\section{Quasi-independence and RSW estimates}
\label{s.qirsw}

In this section we show how the white-noise representation for the field, combined with the Gaussian estimates in the previous section, provide a simple route to establishing quasi-independence for monotonic events. Using the approach of Tassion \cite{tassion2014crossing}, we then deduce RSW estimates at the zero level, i.e.\ the first statement of Theorem \ref{t.main2}. These estimates are also already enough to prove the first statement of Theorem \ref{t.main3}. Recall that we assume throughout that Assumption \ref{a.reg} holds, which guarantees in particular that \eqref{e.beta} holds for a given $\beta > 1$.

\medskip
We begin by stating a general comparison result that is a simple combination of Corollary \ref{c.cameron_martin} and Proposition \ref{p.coupling} (see Definition~\ref{d.sigma_al} for the notation for sub-$\sigma$-algebras):
\begin{proposition}
\label{p.perturbation}
Suppose that \eqref{e.beta} holds for $\beta > 1$, and suppose that there exists a neighbourhood $V$ of $0$ such that $\inf_V |\rho|  > 0$. Then there exist $c_1, c_2> 0$ and $R_0 > 0$ such that, for every $R \ge R_0$, $r \ge 1$ and $t \ge \log R$, every Borel set $D \subset \mathbb{R}^2$ of diameter at most $R$, and every monotonic event $A \in \mathcal{F}_{D}$,
\begin{equation}
\label{e.pqi}
\left| \Pro \left[  f \in A   \right] - \Pro \left[ f_r \in A \right]  \right| \leq c_1 R t \, r^{1-\beta} +  c_1 e^{-c_2 t^2}  .
\end{equation}
\end{proposition}

\begin{proof}
By Proposition \ref{p.coupling}, there exist $c_1, c_2 > 0$ such that,
\[ \Pro \left[ \| f - f_r \|_{\infty,B_R} \geq c_1  t \, r^{1-\beta}  \right] \leq c_1 e^{-c_2 t^2}  ,\]
and applying Corollary \ref{c.cameron_martin} we have result. 
\end{proof}

We now state our main quasi-independence result:

\begin{thm}[Quasi-independence]
\label{t.qi}
Suppose that \eqref{e.beta} holds for $\beta > 1$, and suppose that there exists a neighbourhood $V$ of $0$ such that $\inf_V|\rho|  > 0$. Then there exist $c_1, c_2 > 0$ and $R_0 > 0$ such that, for every $R_1 R_2 \ge R_0$, $r \ge 1$, $t_1 \ge \log R_1$ and $t_2 \ge \log R_2$, every pair of Borel sets $D_1 \subset \mathbb{R}^2$ (resp.\ $D_2$) of diameter at most $R_1$ (resp.\ $R_2$) and such that $r = \text{dist}(D_1, D_2)$, and every pair of monotonic events $A \in \mathcal{F}_{D_1}$ and $B \in \mathcal{F}_{D_2}$,
\begin{multline}
\label{e.qi}
\left| \Pro \left[  f \in A \cap B  \right] - \Pro \left[ f \in A \right] \Pro \left[ f \in B \right] \right|  \le c_1 R_1 t_1 r^{1-\beta} + c_1 e^{-c_2 t_1^2} + c_1 R_2 t_2 r^{1-\beta} + c_1 e^{-c_2 t_2^2} .
\end{multline}
\end{thm}

\begin{remark}
In the special case that $R_1$, $R_2$ and  $r$ are all of the same order $\Theta(R)$, equation~\eqref{e.qi} (with the setting $t_1 = \log R_1$ and $t_2 = \log R_2$) yields a simple bound on 
\begin{equation}
\label{e.qiorder}
 \left| \Pro \left[  f \in A \cap B  \right] - \Pro \left[ f \in A \right] \Pro \left[ f \in B \right] \right|
 \end{equation}
 of order
\[  R^{2-\beta} \log R  .\]
In particular, if \eqref{e.beta} holds for a given $\beta > 2$, then \eqref{e.qiorder} tends to zero as $R \to \infty$; hence our description of \eqref{e.qi} as verifying `asymptotic independence'. 
\end{remark}

\begin{remark}
Theorem \ref{t.qi} can be compared to other similar quasi-independence results that have appeared previously, both in the study of the components of the zero level set of Gaussian entire functions \cite{nsv_2007, nsv_2008, ns_2010}, and also in the study of crossing events for smooth Gaussian fields \cite{bg_16,bmw_17,rv_rsw,bm_17}. 

\medskip
Whereas the approaches to quasi-independence from \cite{bg_16,bmw_17,rv_rsw,bm_17} have proceeded by restricting the field to a lattice and deducing quasi-independence for this discrete field, our approach is to work directly in the continuum, first approximating by a field with exact independence, and then studying of the effect of the approximation on the probability of the events we are interested in. This two step approximation procedure can also be found in \cite{nsv_2007, nsv_2008, ns_2010}, but the approach and the events considered therein are very different.

\medskip
 Moreover, whereas the previous best known sufficient condition for quasi-independence was polynomial decay with exponent $\beta > 4$ \cite{rv_rsw}, we deduce asymptotic independence as long as correlations decay (roughly speaking) polynomially with exponent $\beta > 2$. More precisely (in the case $r=R_1=R_2=R$), \cite{rv_rsw} roughly implies that~\eqref{e.qi} holds for crossing events with the right-hand-side replaced by $c R^{-\beta} \times (\text{Total influence})^2$ where the term `$\text{Total influence}$' is the sum of $R^2$ influences defined in the spirit of the influences from Section~\ref{s.OSSSsub}. In \cite{rv_rsw}, this sum of influences was bound by $R^2$, although heuristics as in Remark~\ref{r.bernou} of the present paper suggest that this sum might be bounded by $R$. In fact, this is the idea that has guided us, even if we have followed a completely different approach to in \cite{rv_rsw}.
\end{remark}

Theorem~\ref{t.qi} is a direct consequence of the following more general proposition: 

\begin{prop}
\label{p.qi}
Suppose that \eqref{e.beta} holds for $\beta > 1$, and suppose that there exists a neighbourhood $V$ of $0$ such that $\inf_V |\rho|  > 0$. Then, there exist $c_1, c_2 > 0$ and $R_0 > 0$ such that, for every $R_1,R_2 \ge R_0$, $r \ge 1$, $t_1 \ge \log R_1$ and $t_2 \ge \log R_2$, every pair of Borel sets $D_1 \subset \mathbb{R}^2$ (resp.\ $D_2$) of diameter at most $R_1$ (resp.\ $R_2$) and such that $r = \text{dist}(D_1, D_2)$, every $n_1,n_2 \in \N$, all sets of monotonic events $A_1, \ldots, A_{n_1} \in \mathcal{F}_{D_1}$ and $B_1, \ldots, B_{n_2} \in \mathcal{F}_{D_2}$, and every event $A$ (resp.\ $B$) in the Boolean algebra generated by $A_1, \ldots, A_{n_1}$ (resp.\ $B_1, \ldots, B_{n_2}$),
\begin{multline}
\label{e.qi2}
\left| \Pro \left[  f \in A \cap B  \right] - \Pro \left[ f \in A \right] \Pro \left[ f \in B \right] \right|\\
\leq c_1 n_1 \left( R_1 t_1 r^{1-\beta} + e^{-c_2 t_1^2} \right)  + c_1 n_2 \left( R_2 t_2 r^{1-\beta}  + e^{-c_2 t_2^2} \right) .
\end{multline}
\end{prop}

\begin{remark}
Although in the present paper we do not need Proposition \ref{p.qi} in full generality, we believe it to be of independent interest, for instance it could be useful if one needs quasi-independence for non-monotonic events which are measurable with respect to a moderate number of monotonic events.
\end{remark}

\begin{proof}[Proof of Proposition~\ref{p.qi}]
Consider the truncated field $f_r$, for which $\{ f_r \in A \}$ is independent of $\{ f_r \in B\}$ by the definition of the events $A$ and $B$, i.e.,
\[  \Pro \left[ f_r \in A \cap B \right] = \Pro \left[ f_r \in A \right] \Pro \left[ f_r \in B \right]   .\]
Then
\begin{align*}
& \left| \Pro \left[ f \in A \cap B \right] - \Pro \left[ f \in A \right] \Pro \left[ f \in B \right] \right|\\
& \qquad = \left| \Pro \left[ f \in A \cap B \right] - \Pro \left[ f_r \in A \cap B \right] + \Pro \left[ f_r \in A \right] \Pro \left[ f_r \in B \right] - \Pro \left[ f \in A \right] \Pro \left[ f \in B \right]  \right| \\
& \qquad \leq \left| \Pro \left[ f \in A \cap B \right] - \Pro \left[ f_r \in A \cap B \right] \right| + \left| \Pro \left[ f_r \in A \right] -  \Pro \left[ f \in A \right] \right| + \left| \Pro \left[ f_r \in B \right] -  \Pro \left[ f \in B \right] \right|
\end{align*}
(indeed, if $a,b,a',b' \in [0,1]$ then $|aa'-bb'| \leq |a-a'|+|b-b'|$). Now, note that
\[
\{ f \in A \cap B \} \bigtriangleup \{ f_r \in A \cap B \} \subseteq \left( \{ f \in A \} \bigtriangleup \{ f_r \in A \} \right) \cup \left( \{ f \in B \} \bigtriangleup \{ f_r \in B \} \right) ,
\]
and so
\[
\left| \Pro \left[ f \in A \cap B \right] - \Pro \left[ f \in A \right] \Pro \left[ f \in B \right] \right| \leq 2 \Pro \left[ f \in A \bigtriangleup f_r \in A \right] + 2 \Pro \left[ f \in B \bigtriangleup f_r \in B \right] .
\]
Similarly, note that
\[
\Pro \left[ f \in A \bigtriangleup f_r \in A \right] \leq \sum_{i=1}^{n_1} \Pro \left[ f \in A_i \bigtriangleup f_r \in A_i \right] , 
\]
and analogously for $\Pro \left[ f \in B \bigtriangleup f_r \in B \right]$. As a result
\[
\left| \Pro \left[ f \in A \cap B \right] - \Pro \left[ f \in A \right] \Pro \left[ f \in B \right] \right| \leq 2 \sum_{i=1}^{n_1} \Pro \left[ f \in A_i \bigtriangleup f_r \in A_i \right] + 2 \sum_{j=1}^{n_2} \Pro \left[ f \in B_j \bigtriangleup f_r \in B_j \right] .
\]
Applying Proposition \ref{p.perturbation}, and since $A$, $B$ are monotonic, we have
\[
\Pro \left[ f \in A_i \bigtriangleup f_r \in A_i \right] \leq c_1 R_1 t_1  r^{1-\beta} + c_1 e^{-c_1 t_1^2} ,
\]
and similarly for $\Pro \left[ f \in B_j \bigtriangleup f_r \in B_j \right]$, which completes the proof.
\end{proof}

To finish the section we use Theorem \ref{t.qi} to deduce RSW estimates at the zero level $\ell =0$ in the case that, additionally, Assumption~\ref{a.sym} and the weak positivity condition in Assumption~\ref{a.swp} holds. For this we rely on the strategy of~\cite{tassion2014crossing} (see Section~4 of~\cite{rv_rsw} for more detail). These estimates constitute the first statement of Theorem~\ref{t.main2}.

\begin{thm}[RSW estimates]
\label{t.rsw}
Suppose that Assumption \ref{a.sym} holds, that the weak positivity condition in Assumption \ref{a.swp} holds, and that Assumption \ref{a.dec} holds for a given $\beta > 2$. Then for each quad $Q$,
\[     \inf_{R > 0} \, \mathbb{P}[ f \in \cross_0(RQ) ]    > 0  \quad \text{and} \quad    \sup_{R > 0}  \, \mathbb{P} [f \in \cross_0(RQ) ]  < 1,  \]
and moreover there exist $c, d > 0$ such that, for each $1 \leq r \leq R$,
\[   \mathbb{P} \left[   f   \in \arm_0(r, R  ) \right]     <  c \left( \frac{r}{R} \right)^{d}.  \] 
\end{thm}
\begin{proof}
Given that quasi-independence holds by Theorem \ref{t.qi} (recall that the fact that there exists a neighbourhood $V$ of $0$ such that $\inf_V |\rho|  > 0$ is implied by weak positivity and $\beta > 2$; see Proposition \ref{p.aswp}), this follows directly from the arguments in \cite{tassion2014crossing}. More precisely, Tassion's argument relies on three conditions being satisfied (see~\cite[Section 4]{rv_rsw} and \cite[Section 4.2]{bg_16} for details): 
\begin{enumerate}
\item The FKG inequality for crossing-type events (that holds since $\kappa \geq 0$); 
\item Sufficient symmetry, which is guaranteed by Assumption \ref{a.sym}; and 
\item At only one place in the proof, the following quasi-independence property (see \cite[Lemma 4.3]{rv_rsw}): For every $\delta > 0$ and $C>0$ there exists $R_0 > 0$ such that, for every $R > R_0$, every pair $D_1$ and $D_2$ of Borel subsets of the plane of diameter at most~$CR$ and at distance at least $R$ from each other, and every pair of monotonic events $A \in \mathcal{F}_{D_1}$ and $B \in \mathcal{F}_{D_2}$,
\begin{equation} \label{e.tassion}
\left| \Pro \left[ f \in A \cap B \right] - \Pro \left[ f \in  A \right] \Pro \left[ f \in B \right] \right| \leq \delta .
\end{equation}
\end{enumerate}
Since \eqref{e.tassion} is implied by Theorem~\ref{t.qi} (recall that we assume $\beta > 2$), these conditions are satisfied and Tassion's arguments are valid.
\end{proof}

\begin{remark}
As shown in \cite[Section 4.2]{bg_16}, in light of Theorem \ref{t.rsw} and the quasi-independence in Theorem \ref{t.qi}, the zero level set~$\mathcal{L}_0$ also satisfies equivalents of the RSW estimates.
\end{remark}

We finish this section by showing how to deduce the first statement of Theorem \ref{t.main3} from the RSW estimates and our analysis in Section \ref{s.gauss}.

\begin{proof}[Proof of the first statement of Theorem \ref{t.main3}]
Let $c_1 > 1$ be given. By Corollary \ref{c.cameron_martin}, for each quad $Q$ there exists a constant $c_2 > 0$ such that for each $R > 0$,
\[ 0 \le   \mathbb{P}[ \cross_0(RQ) ] -  \mathbb{P}_{R^{-c_1}}[ \cross(RQ ) ]  < c_2 R^{1-c_1}, \]
and so, in particular, as $R \to \infty$,
\[  \left| \mathbb{P}[ \cross_0(RQ) ] -  \mathbb{P}[ \cross_{R^{-c_1}}(RQ ) ]  \right| \to 0 .  \]
Combining  with the RSW estimates in Theorem \ref{t.rsw}, we have the result.
\end{proof}

\bigskip
\section{A first description of the phase transition}
\label{s.osss}

In this section we begin our study of the sharp phase transition, establishing a first description of the phase transition for a single rectangle at levels $\ell > 0$ which are polynomially small in the scale of the rectangle; in the final section we will bootstrap this to complete the main results. Throughout this section we suppose that Assumptions \ref{a.reg} and \ref{a.sym} hold, that the strong positivity condition in Assumption~\ref{a.swp} holds, and that Assumption~\ref{a.dec} holds for a given $\beta > 2$ that is henceforth fixed.

\medskip
Let us begin by introducing streamlined notation for crossing events involving rectangles. For $\rho_1,\rho_2 > 0$, let $\cross_\ell(\rho_1,\rho_2)$ denote the crossing event $\cross_\ell(Q)$ in the case that $Q = (D; \gamma, \gamma')$, where $D$ is the rectangle $[0,\rho_1] \times [0,\rho_2]$ and $\gamma$ and $\gamma'$ are respectively the left and right sides $\{0\} \times [0, \rho_2]$ and $\{\rho_1\} \times [0, \rho_2]$. We also define $\cross_\ell^\ast(\rho_1,\rho_2)$ for the crossing event $\cross^\ast_\ell(Q)$ introduced in Section \ref{s.gauss} that corresponds to this quad.

\medskip
In this section we will identify $f = f_r$ for the setting~$r = \infty$. The main result of the section is the following:

\begin{theorem}
\label{t.phase_trans}
There exists a constant $\theta > 0$ such that, as $R \to \infty$, 
\[ \inf_{\bar r \in [ R^{\theta}, \infty] }  \Pro \left[ f_{\bar r} \in \cross_{R^{-\theta}}(2R,R) \right] \to  1  . \]
\end{theorem}

As described in Section \ref{s.overview}, we prove Theorem \ref{t.phase_trans} by applying the OSSS inequality to (the white-noise representation of) the truncated discretised field $f^\eps_{r}$ and the complement of the event $\cross^\ast_{\ell}(2R, R)$, where $r = R^{h}$ and $\eps = R^{-\gamma}$ for $h, \gamma > 0$ well-chosen constants. While it would have been more natural to work with the event $\cross_{\ell}(2R, R)$, and although we expect that the events  $\{f_r^\eps \notin \cross^*_{\ell}(2R,R)\}$ and $\{f_r^\eps \in  \cross_{\ell}(2R,R)\}$ are equal almost surely (as is the case for $f$ and $f_r$ for instance; see Proposition \ref{p.areg}), since we lack a proof of this (the field $f_r^\eps$ is degenerate and is not stationary, so standard arguments do not apply), we must take care to distinguish these events.

\subsection{Sprinkling}
\label{s.infl}
We begin the proof of Theorem \ref{t.phase_trans} with a `sprinkling' procedure that yields RSW-type estimates for the excursion set $\mathcal{E}_\ell$ of $f_r^\eps$ at polynomially-small levels.

\medskip
For each $0 < \rho_1 < \rho_2$ and $x \in \mathbb{R}^2$, let $\arm_\ell(x;\rho_1,\rho_2)$ (resp.\ $\arm_\ell^*(x;\rho_1,\rho_2)$) denote the event that there is a connected component of $\mathcal{E}_\ell$ (resp.\ $\mathcal{E}^c_\ell$) that intersects both $\partial B_{\rho_1}(x)$ and $\partial B_{\rho_2}(x)$. 

\begin{prop}\label{p.sprink}
We have the following two `sprinkling' properties:
\begin{enumerate}
\item[i)] Let $d$ be the constant appearing in Theorem \ref{t.rsw}. There exist constants $c_1,c_2>0$ such that, for each $h,\gamma > 0$ and each $\theta \in (0,\min \{ \gamma,(\beta-1)h \}]$, there is a constant $R_0 \geq 1$ such that the following holds: For every $R \geq R_0$, $1 \leq \rho_1 \leq \rho_2 \leq R$, $x \in \R^2$, and $\ell \ge R^{-\theta}$,
\[ \Pro \left[  f_{R^h}^{R^{-\gamma}} \notin  \cross^*_{\ell}(2R,R) \right] > c_1 \quad \text{and} \quad \Pro \left[ f_{R^h}^{R^{-\gamma}} \in \arm^\ast_{\ell}(x;\rho_1,\rho_2) \right] < c_2 \left( \frac{\rho_1}{\rho_2} \right)^{d}   . \]
\item[ii)] Moreover, for each $h,\gamma > 0$ and each $\theta \in (0,\min \{ \gamma,(\beta-1)h \}]$,
\[ \sup_{\ell \geq R^{-\theta}} \sup_{\bar r \in [ R^h, \infty] }   \Pro \left[   \{ f_{R^h}^{R^{-\gamma}} \notin \cross^\ast_{\ell}(2R,R) \} \setminus  \{f_{\bar r} \in \cross_{2 \ell}(2R,R) \} \right] \underset{R \rightarrow \infty}{\longrightarrow} 0  . \]
\end{enumerate} 
\end{prop}

\begin{proof}
The proof is based on the simple fact that, if $A$ is an increasing event that depends only on $B_R$, and if $f$ and $g$ are random fields satisfying
\[ \Pro[ f \in A ] > c \quad \text{and} \quad \Pro [ \|f - g \|_{\infty,B_R} > \ell ] < \delta \]
for some constants $c, \ell, \delta > 0$, then $\Pro[g+ \ell \in A] > c - \delta$, and similarly for $A$ decreasing.

\medskip
Now let $h, \gamma > 0$ be given. By Proposition~\ref{p.coupling} there exist $c_3, c_4 > 0$ such that, for each $R \ge 1$,
\[ \sup_{ \bar r \in [R^h, \infty] }  \Pro \left[ \| f_{\bar r} - f^{R^{-\gamma}}_{R^h} \|_{\infty,B_R} > c_3 (\log R) ( R^{-(\beta-1)h} + R^{-\gamma} )  \right] < c_3  e^{-c_4(\log R)^2}  . \]
Since $\theta < \min\{ \gamma, (\beta-1)h \}$, this implies that, for all $\ell \ge R^{-\theta}$, 
\[ \sup_{ \bar r \in [R^h, \infty] } \Pro \left[ \| f_{\bar r}  - f^{R^{-\gamma}}_{R^h} \|_{\infty,B_R} > \ell  \right] < c_5  e^{-c_6(\log R)^2}  \]
for constants $c_5, c_6 > 0$. Since the right-hand side of the above tends to zero as $R \to \infty$ faster than any polynomial, and since $\cross^\ast_\ell(\rho_1,\rho_2)$ and $\arm^*_{\ell}(x; \rho_1,\rho_2)$ are decreasing events, combining with the RSW estimates in Theorem~\ref{t.rsw} gives the first two results as well as the third result with $ \{f_{\bar r} \in \cross_{2 \ell}(2R,R) \}$ replaced by $ \{f_{\bar r} \notin \cross^\ast_{2 \ell}(2R,R) \}$. In light of the third statement of Proposition \ref{p.areg} we are done.
\end{proof}

\subsection{Connecting Russo's formula to the OSSS influences}
\label{s.connect}
The next step is to give a Russo-type formula that is applicable in our setting, and then to show that the `influences' that appear in this formula are comparable to the influences that appear in the OSSS inequality (see Section \ref{s.overview} where we introduce this inequality). 

\medskip
We begin by considering the case of a standard $n$-dimensional Gaussian vector $X$, for which we have the following Russo-type formula: For every Borel set $A \subset \R^n$ and $\ell \in \mathbb{R}$,
\begin{equation} \label{e.russo_for_iid_Gaussian}
\frac{d}{d\ell} \Pro \left[ (X_1+\ell,\cdots,X_n+\ell) \in A \right] = \sum_{i=1}^n \mathbb{E}[X_i \id_{  (X_1+\ell,\cdots,X_n+\ell) \in A } ] .
\end{equation}
 By analogy with the Boolean Russo formula \eqref{e.rus1}, the summands $\E \left[ X_i \un_A \right]$ can be considered as the `influence' of each coordinate $i$ on the event $A$; in the case that $A$ is increasing, these are always positive. 
 
 \medskip
 We next connect the above notion of influence to the `resampling' influences $I_i$ that appear in the OSSS inequality, defined in Section \ref{s.osss}. For each increasing Borel set $A \subseteq \R^n$ and $i \in \{ 1, \cdots, n \}$, let $Y_i^A \in [-\infty,\infty]$ be the random variable, depending on all coordinates except the $i^{\rm{th}}$, such that, for every $x < Y_i^A$, $(X_1,\cdots,X_{i-1},x,X_i,\cdots,X_n) \notin A$ and, for every $x > Y_i^A$, $(X_1,\cdots,X_{i-1},x,X_i,\cdots,X_n) \in A$; in other words, $Y_i^A$ is the \textit{threshold} for the event $A$ with respect to the $i^{\rm{th}}$ coordinate. Then, by symmetry, 
\[ \E \left[ X_i \un_{X \in A} \right] = \E \left[ X_i \un_{X_i \geq Y_i^A} \right] = \E \left[ X_i \un_{X_i \geq |Y_i^A|} \right] . \]
Now, let $\widetilde{X}=X$ except that the $i^{th}$ coordinate is resampled independently. Then we can define the influence of each $i \in \{ 1, \cdots, n \}$ on $A$ as in Section~\ref{s.OSSSsub}, i.e.,
\[
I_i(A):=I_i^{\mathcal{N}(0,1)}(A)=\Pro \left[ \un_{A}(X) \neq \un_A(\widetilde{X}) \right] ,
\]
(we shall use the abbreviation $I_i:=I_i^{\mathcal{N}(0,1)}$ throughout Section~\ref{s.osss}). Note that we have
\begin{align}
\label{e.con}
  I_i(A)  = 2 \Pro \left[ X_i \leq Y_i^A \leq \widetilde{X}_i \right] & \leq 2 \Pro \left[ X_i \geq |Y_i^A| \right] \\
 \nonumber  &\leq c_\text{Rus} \, \E \left[ X_i \un_{X_i \geq |Y_i^A|} \right] = c_\text{Rus} \, \E \left[ X_i \un_{X \in A} \right] , 
\end{align}
where $c_\text{Rus} > 0$ denotes the absolute constant
\begin{equation}
\label{e.crus}
c_\text{Rus} =\sup_{a \geq 0} \Pro \left[ Z \geq a \right] / \E \left[ Z \un_{Z \geq a} \right] < \infty , 
\end{equation}
 for $Z$ a standard normal random variable. In other words, the `influences' in Russo's formula are comparable to the resampling influences, just as they are in the Boolean setting of percolation (as discussed in Section \ref{s.overview}). Note that the fact that $A$ was increasing was crucial in attaining~\eqref{e.con}.

\medskip
We now return to the setting of smooth Gaussian fields. Fix $\eps > 0$ and consider the discretised field $f^\eps = q \star W^\eps$, where for the purposes of this discussion we assume only that Assumption~\ref{a.dec} holds for $\beta > 1$ rather than $\beta > 2$ (we also do not need Assumption \ref{a.sym} here). Let $r \geq 1$, let~$D$ be a bounded Borel subset of $\R^2$ and let $\mathcal{D} = \{ x \in \Z^2 \, : \, \dist(x,D) \leq r \}$, where $\dist$ is the Euclidean distance. Also, let $A \in \mathcal{F}_D$ (see Definition~\ref{d.sigma_al} for the notations of $\sigma$-algebras) and, for each $\ell \in \R$, let $\widetilde{A}_\ell$ be the Borel subset of $\R^{\calD}$ such that
\[ \{ f_r^\eps + \ell \in A \} = \{ (\eta_v)_{v \in \calD } \in \widetilde{A}_\ell \}  .  \]
Observe that, if $A$ is increasing and $q \ge 0$, $\widetilde{A}_\ell$ is increasing with respect to $(\eta_v)_{v \in \calD}$. This observation allows us to deduce a Russo-type formula in terms of the `resampling' influences that appear in the OSSS inequality:

\begin{proposition}[Russo's formula in terms of the OSSS influences]
\label{p.russo}
Fix $\eps > 0$, let $D$ be a bounded Borel subset of $\R^2$ and let $A \in \mathcal{F}_D$ be increasing. Define $\mathcal{D}$ and $\widetilde{A}_\ell$ as above. Then, for each $\ell \in \mathbb{R}$,
\begin{equation}
\label{e.rus2}
\frac{d}{d\ell} \Pro \left[ f^\eps_r + \ell \in A \right] =  \frac{  \eps}{ \| q_r \|_{L^1} }  \sum_{v \in \calD} \E[ \eta_v  \id_{ f_r^\eps + \ell \in A} ]  \ge \frac{c_\text{Rus}  \, \eps}{\|q\|_{L^1}} \sum_{v \in \calD} I_v( \widetilde{A}_\ell) ,
\end{equation}
where $c_\text{Rus} > 0$ is the absolute constant defined in \eqref{e.crus}.
\end{proposition}

\begin{proof}
First recall that
\[ f_r^\eps(x) = \frac{1}{\eps} \sum_{v \in \mathcal{D}} \eta_v \int_{u \in v+[-\eps/2,\eps/2]^2} q_r(x-u)  \, du  ,\]
and that $q_r \in L^1$. Hence, for every $x \in \R^2$,
\begin{eqnarray*}
f_r^\eps(x) + \ell & = & \frac{1}{\eps} \sum_{v \in \mathcal{D}} \left( \eta_v+\ell \frac{\eps}{ \int_{\R^2} q_r(x-u) \, du}  \right) \int_{u \in v+[-\eps/2,\eps/2]^2} q_r(x-u) \, du\\
& = & \frac{1}{\eps} \sum_{v \in \mathcal{D}} \left( \eta_v+\ell \frac{\eps}{ \| q_r \|_{L^1} }  \right) \int_{u \in v+[-\eps/2,\eps/2]^2} q_r(x-u) \, du  .
\end{eqnarray*}
As a result, we have
\begin{eqnarray*}
\{  f_r^\eps + \ell \in A \} & = & \left\lbrace \left(\eta_v+\ell \frac{\eps}{\| q_r \|_{L^1} } \right)_{v \in \calD} \in \widetilde{A}_0  \right\rbrace  ,
\end{eqnarray*}
and so, by \eqref{e.russo_for_iid_Gaussian}, we have the equality in~\eqref{e.rus2}. The inequality then follows from \eqref{e.con} (and the bound $\|q_r\|_{L^1}  \le \|q\|_{L^1}  $). Indeed, the fact that $A$ is increasing and that $q \geq 0$ imply that $\widetilde{A}_\ell$ is increasing for every $\ell$. Note that this is the only part of the paper where the strong positivity condition $q \ge 0$ is required.
\end{proof}

\subsection{Applying the OSSS inequality}

We are now in a position to apply the OSSS inequality, and complete the proof of Theorem~\ref{t.phase_trans}. 

\begin{proof}[Proof of Theorem~\ref{t.phase_trans}]

Let $d > 0$ denote the constant appearing in Proposition \ref{p.sprink}; we can and will assume that $d<1$. Fix $\gamma, h \in (0,1)$ such that  $\gamma < d(1-h)$, and fix a $\theta > 0$ such that
\[   \theta < \min\{ \gamma, (\beta - 1) h ,  d(1-h) - \gamma \} , \]
which satisfies in particular the restriction on $\theta$ in the statement of Proposition~\ref{p.sprink}. For the remainder of the proof we abbreviate $\eps=\eps(R) = R^{-\gamma}$ and $r = r(R) = R^{h}$.

\medskip
In the sequel we let $\Omega(1)$ and $\grandO{1}$ denote constants, that may depend on the parameters, but with the properties that (i) $\Omega(1)$ is bounded away from zero in $R$, and $\grandO{1}$ is bounded above in $R$, and (ii) $\Omega(1)$ and $\grandO{1}$ are independent of $\ell$.

\medskip
Define the (finite) set of vertices 
\[ {\calD}_R =\eps \Z^2 \cap [-r,2R+r] \times [-r,R+r]
.\] 
For every $\ell \in \R$ and $R \ge 1$, let $\widetilde{\cross}_\ell(2R,R)$ be the increasing Borel subset of $\R^{\calD_R}$ such that
\[
\{ f_r^\eps \notin \cross^\ast_{\ell} (2R, R) \}  = \{ (\eta_v)_{v \in \calD_R} \in \widetilde{\cross}_\ell(2R,R) \} .
\] 
Our strategy will be to apply OSSS to the event $\{f_r^\eps \notin \cross^\ast_\ell(2R,R)\}$; this part of the argument is valid for all $R \ge 1$ and $\ell \in \mathbb{R}$. By the discussion in Section~\ref{s.infl}, we have the following differential formula in terms of the resampling influences $I_v$:

\begin{claim}
\label{c.russo}
For every $\ell \in \R$,
\begin{equation}
\label{e.rus3}
\frac{d}{d\ell} \Pro \left[ f_r^\eps \notin  \cross^\ast_{\ell}(2R,R) \right]   \geq  \Omega(1) \, \eps \sum_{v \in \calD_R} I_v(\widetilde{\cross}_\ell(2R,R)) .
\end{equation}
\end{claim}
\begin{proof}
Since $q_r \ge 0$ and $\| q_r \|_{L^1}  \le \|q \|_{L^1}   < \infty$ (since Assumption~\ref{a.dec} holds for $\beta > 2$), this is an application of Proposition \ref{p.russo} to the field $f_r^\eps$ and the event $A = \cross_{0}(2R,R)$.
\end{proof}

Let us now apply the OSSS inequality to the right-hand side of \eqref{e.rus3}. For this we define an algorithm $\mathcal{A}$ that reveals sequentially the values of $(\eta_v)_{v \in \calD_R}$ such that (i) $\mathcal{A}$ determines the event $\{f_r^\eps \notin \cross^\ast_\ell(2R,R)\}$, and (ii) the revealment $\delta_v(\mathcal{A})$ for every vertex $v \in \calD_R $ is small. Our algorithm is adapted from \cite{ahlberg2017noise} (using ideas that take their root in~\cite{benjamini1999noise} and~\cite{schramm2010quantitative}); the basic idea is to explore each of the connected components of $ \mathcal{E}_\ell^c$ that intersects a (randomly chosen) horizontal line-segment $L$ across the rectangle $[0, 2R] \times [0, R]$, and determine whether any of these components join the top and bottom sides of the rectangle, and hence validate the event $\{f_r^\eps \in \cross^\ast_\ell(2R,R)\}$ (see Figure \ref{f.algo} for an illustration).

\medskip
Abbreviate $D = [0, 2R] \times [0, R]$. We call a point $x \in D$ \textit{safe} if every vertex $v \in \calD_R$ within a distance $r$ of $x$ has been revealed (of course, this set will change during the running of the algorithm, but it is always non-decreasing), and let $\mathcal{S} \subset D$ denote the set of safe points. We call a path in $D$ a \textit{safe blocking path} if it lies in $\mathcal{S} \cap \mathcal{E}_\ell^c$; note that since $f^\eps_r$ is an $r$-dependent field, the value of the field $f^\eps_r$ is known precisely on the set of safe points $\mathcal{S}$, so the existence of a safe blocking path is measurable with respect to the revealed $\eta_v$.

\medskip
The algorithm $\mathcal{A}$ is defined as follows:
\medskip
\paragraph{\textbf{Algorithm $\calA$:}} 
\begin{enumerate}
\item Initialise a random seed $k \in \{0, 1, \ldots , \lfloor R/r \rfloor \}$, selected uniformly at random.
\item Reveal (in arbitrary order) the value of $\eta_v$ for every $v \in \calD_R$ at distance at most $2r$ from the horizontal line segment $L := [0, 2R] \times \{kr\}$.
\item Iterate the following steps:
\begin{enumerate}
\item If there is a safe blocking path between the `bottom' side $\nu := [0, 2R] \times \{0\}$ and the `top' side $\nu' := [0, 2R] \times \{R\}$, terminate with output~$0$.
\item Identify the subset $\mathcal{U} \subset \partial \mathcal{S}$ such that there is a safe blocking path between $L$ and $\partial \mathcal{S}  \setminus (\nu \cup \nu') $. If the subset $\mathcal{U}$ is empty, terminate with output $1$.
\item Reveal (in arbitrary order) the value of $\eta_v$ for every $v \in \calD_R$ at distance at most $2r$ from each point in $\mathcal{U}$ that has not yet been revealed.
\end{enumerate}
\end{enumerate}

\begin{figure}[h!]
\centering
\includegraphics[scale=0.7]{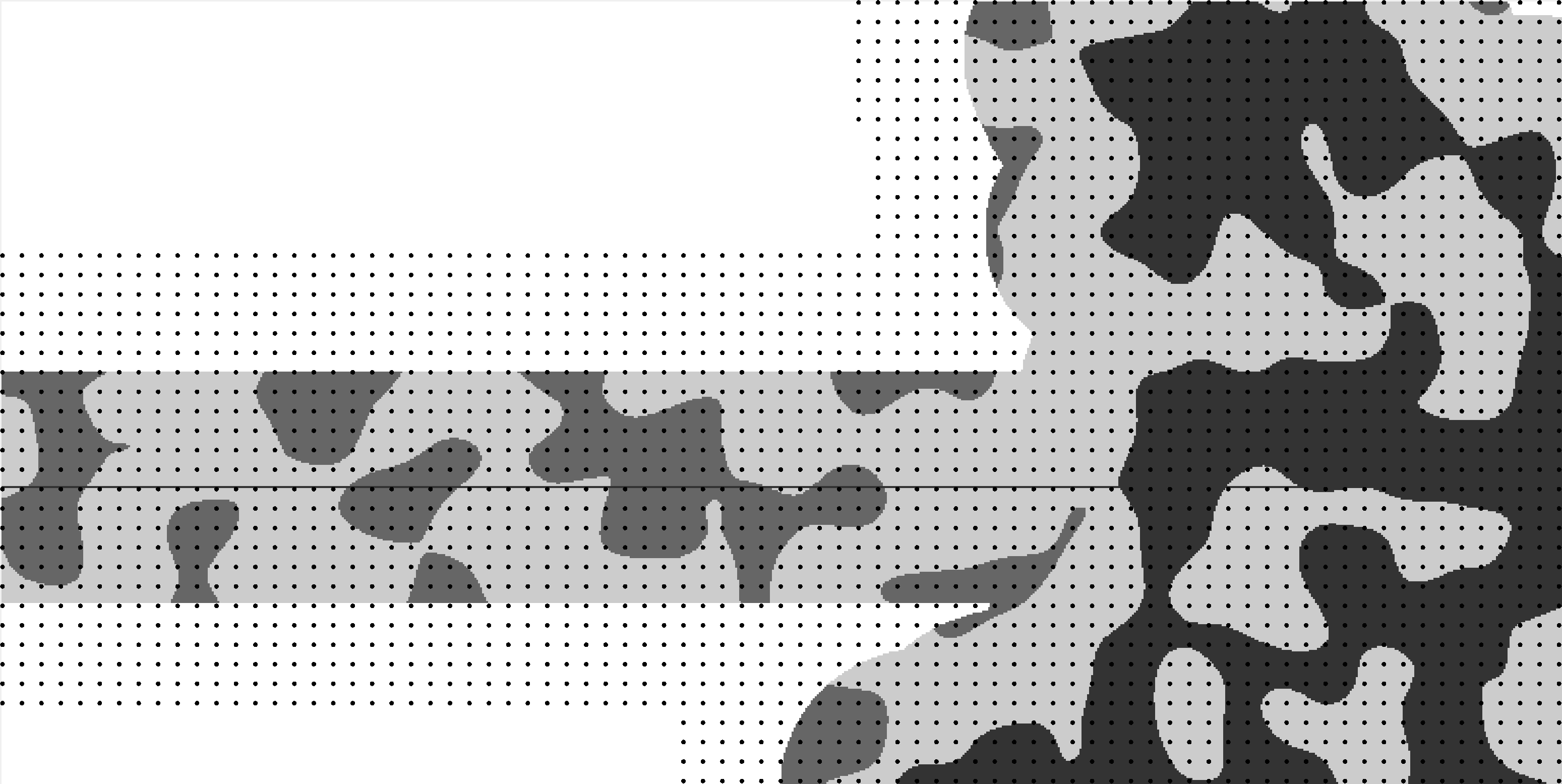}
\caption{An illustration of a run of algorithm $\mathcal{A}$ showing (i) the horizontal line~$L$, (ii) the white-noise coordinates $\eta_v$ that were revealed by the run, and (iii) the `safe' set $\mathcal{S}$ at the end of the run, which consists of (a) the `safe' subset of $\mathcal{E}_\ell$ (in light grey), and (b) the `safe' subset of $\mathcal{E}_\ell^c$ (in black and dark grey). For this run $\mathcal{A}$ terminated with output $0$, indicating that $\{f_r^\eps \in \cross^\ast_\ell(2R,R)\}$ occurred, since there is `safe blocking path' (in black) between $\nu$ and~$\nu'$. Credit: Dmitry Beliaev}\label{f.algo}
\end{figure}

Observe that algorithm $\mathcal{A}$ determines the event $\{f_r^\eps \notin \cross^\ast_\ell(2R,R)\}$ (i.e.\ it terminates with output value $\id_{f_r^\eps \notin \cross^\ast_\ell(2R,R) }$), since either (i) $\mathcal{A}$ reveals a safe blocking path in $\mathcal{E}_\ell^c \cap D$ which connects the `bottom' side $\nu$ and the `top' side $ \nu'$, in which case the output value is $0$ and $\{f_r^\eps \in \cross^\ast_\ell(2R,R)\}$ occurs, or else (ii) $\mathcal{A}$ terminates with output value $1$ and reveals there to be no path in $\mathcal{E}_\ell^c \cap D$ which connects $\nu$ and $\nu'$; in this case we deduce that $\{f_r^\eps \notin \cross^\ast_\ell(2R,R)\}$ occurs.

\medskip
Next observe that the revealments under this algorithm are bounded above by
\begin{equation}
\label{e.reveal}
\max_{v \in \calD_R} \delta_v(\mathcal{A}) \le    \grandO{1}  \,  \lfloor R/r  \rfloor^{-1}  \bigg( 1 + \sum_{k=1}^{\lfloor R/r \rfloor  }   \Pro \left[ f^\eps_r \in \arm_\ell^*(v; 2r,k r) \right]  \bigg) , 
\end{equation}
since a vertex $v$ is revealed if and only if either (i) $\text{dist}(v, L) \le 2r$, or (ii) there is a connected component of $\mathcal{E}_\ell^c$ that intersects both $L$ and $B_v(2r)$, which implies the existence of the one-arm event  $\arm_{\ell,r}^{*,\eps}(v; 2r, \text{dist}(v, L)) $ (here $\text{dist}(v, L)$ is the vertical distance between $v$ and the line $L$).

\medskip
Combining Claim \ref{c.russo}, the OSSS inequality in Theorem \ref{t.osss} and the bound on the revealments in \eqref{e.reveal}, we obtain 
\begin{equation}
\label{e.derbound}
 \frac{d}{d\ell} \Pro \left[ f_r^\eps \notin \cross_\ell^\ast(2R,R) \right]  \ge  \frac{  \Omega(1) \,  \eps \var \left( f^\eps_r \notin \cross^{\ast}_{\ell}(2R,R) \right)  }{     \lfloor R/r  \rfloor^{-1}  \left(1 +   \sum_{k=1}^{\lfloor R/r \rfloor  }   \Pro \left[ f^\eps_r \in \arm_\ell^*(v; 2r,k r) \right]    \right)} .  
 \end{equation}

\medskip
At this point we restrict the analysis to $R$ sufficiently large and levels $\ell \ge R^{-\theta}$. With the bound on arm events given in Proposition~\ref{p.sprink}, there exists $R_0 > 0$ such that, if $R \geq R_0$ and $\ell \ge R^{-\theta}$,
\begin{align}
\label{e.derbound2}
 \lfloor R/r  \rfloor^{-1} \bigg(1 +   \sum_{k=1}^{\lfloor R/r \rfloor  }   \Pro \left[ f^\eps_r \in \arm_\ell^*(v; 2r,k r) \right] \bigg)  & \le    \grandO{1}  \,  \lfloor R/r  \rfloor^{-1} \bigg( 1 + \sum_{k=1}^{\lfloor R/r \rfloor }   k^{-d} \bigg)  \\
 \nonumber & \le  \grandO{1} \, \lfloor R/ r \rfloor^{-d }  \le \grandO{1} \, R^{-d(1-h) } \, .
 \end{align}
Since $\eps=R^{-\gamma}$, equations \eqref{e.derbound} and \eqref{e.derbound2} imply that
\[ \frac{d}{d\ell} \Pro \left[ f_r^\eps \notin \cross^\ast_\ell(2R,R) \right]  \ge  \Omega(1) \,  R^{d(1-h)-\gamma}  \var \left( f_r^\eps \notin \cross^\ast_\ell(2R,R)\right) \,  .  \]
Now, let $\delta > 0$ and let us show that, if $R$ is sufficiently large, then
\begin{equation}\label{e.last_thing_to_prove}
\Pro \left[ f_r^\eps \notin \cross^\ast_\ell(2R,R)  \right] \geq 1-\delta .
\end{equation}
Note that combined with the final statement of Proposition~\ref{p.sprink}, this is enough to conclude the proof of Theorem~\ref{t.phase_trans} (with $4R^{-\theta}$ instead of $R^{-\theta}$ but then one can replace $\theta$ by $\theta + a$ for any $a>0$). 

\medskip
We prove~\eqref{e.last_thing_to_prove} as follows. First we note that, thanks to the first statement of Proposition \ref{p.sprink}, if $R$ is sufficiently large and $\ell \ge R^{-\theta}$ then
\[ \var \left( f_r^\eps \notin \cross^\ast_\ell(2R,R)  \right) \geq \Omega(1) \,  \left(1-\Pro \left[ f_r^\eps \notin \cross^\ast_\ell(2R,R)\right] \right)  . \]
Now, assume that $R$ is such that $\Pro \left[ f_r^\eps \notin \cross^\ast_\ell(2R,R) \right] < 1-\delta$. It is then sufficient to prove that this implies that $R$ cannot be too large. For this purpose, we note that this implies that
\[ \var \left( f_r^\eps \notin \cross^\ast_\ell(2R,R)\right) > \Omega(1) \delta  . \]
Hence
\[ \frac{d}{d\ell} \Pro \left[f_r^\eps \notin \cross^\ast_\ell(2R,R) \right] > \Omega(1) \delta R^{d(1-h)-\gamma}  . \]
Integrating from $R^{-\theta}$ to $2R^{-\theta}$, we obtain that
\[\Pro \left[ f_r^\eps \notin \cross^\ast_\ell(2R,R) \right] > \Omega(1) \delta R^{d(1-h)-\gamma-\theta}  .\]
Since $\Pro \left[ f_r^\eps \notin \cross^\ast_\ell(2R,R) \right] \leq 1$ and since $d(1-h)-\gamma-\theta>0$, this implies that $R$ cannot be too large and we have proved~\eqref{e.last_thing_to_prove}.

\end{proof}

\begin{remark}
\label{r.3d}
In~\cite{duminil2017sharp,duminil2017exponential,duminil2018poisson} the OSSS inequality was applied to one-arm events rather than to crossing events as we do here; for many discrete models, this yields a differential inequality that implies the sharpness of phase transition in any dimension. While it is possible such an approach would also work in our setting, it seems likely that it would require new ideas to implement. The main difficulty comes from the fact that, if we want the algorithm $\mathcal{A}$ to determine the value of $f^\eps_r(x)$ for some $x$, then $\mathcal{A}$ must reveal \textit{all} of the white-noise coordinates in the ball $B_r(x)$ of \textit{growing} radius $r \gg 1$, which results in a differential inequality that is not strong enough to deduce the phase transition.
\end{remark}

\bigskip
\section{Proof of the main results}
\label{s.final}

The remainder of our results can all be deduced from Theorem~\ref{t.phase_trans} via classical gluing and bootstrapping techniques. Since some of the arguments in this section are rather standard, we skip many of the details and instead refer to relevant literature. As in Section \ref{s.osss}, for every $\rho_1,\rho_2 > 0$, we write $\cross_\ell(\rho_1,\rho_2)$ for the crossing event $\cross_\ell(Q)$ in the case that $Q = (D; \gamma, \gamma')$, where $D$ is the rectangle $[0,\rho_1] \times [0,\rho_2]$ and $\gamma$ and $\gamma'$ are respectively the left and right sides $\{0\} \times [0, \rho_2]$ and $\{\rho_1\} \times [0, \rho_2]$.

\medskip
The proof of the second statement of Theorem~\ref{t.main3} is straightforward:

\begin{proof}[Proof of the second statement of Theorem~\ref{t.main3}]
For the events $\cross_\ell(2R,R)$, the required statement is a direct consequence of Theorem~\ref{t.phase_trans}. Standard gluing arguments (see \cite[Section 4.2]{bg_16} for details) allow the conclusion to be extended to every quad.
\end{proof}

The main technical novelty in this section is to deduce from Theorem \ref{t.phase_trans} a version of the third statement of Theorem \ref{t.main2} for the rectangle $[0, 2] \times [0, 1]$. 

\begin{theorem}
\label{t.quantfor31}
Suppose that Assumptions \ref{a.reg}--\ref{a.sym} hold, that the strong positivity condition in Assumption \ref{a.swp} holds, and that Assumption~\ref{a.dec} holds for a given $\beta > 2$. Then for every $\ell > 0$ there exist $c_1, c_2 > 0$ such that, for all $R \ge 1$,
\[\Pro \left[ f \in \cross_\ell(2R, R)  \right] > 1- c_1 e^{-c_2R} . \]
\end{theorem}

Before proving Theorem \ref{t.quantfor31} we state two auxiliary result. The first result is a kind of `sprinkled' quasi-independence statement that we deduce directly from Proposition \ref{p.coupling}.

\begin{proposition}
\label{p.sqi}
Suppose that Assumption~\ref{a.dec} holds for a given $\beta > 2$, and suppose that there exists a neighbourhood $V$ of $0$ such that $\inf_V|\rho|  > 0$. Then, there exist $c_1, c_2 > 0$ such that, for every $R \ge 1$, every pair of Borel sets $D_1 \subset \mathbb{R}^2$ (resp.\ $D_2$) of diameter at most $5R$ and such that $\text{dist}(D_1, D_2) \ge \sqrt{R}$, and every pair of decreasing events $A \in \mathcal{F}_{D_1}$ and $B \in \mathcal{F}_{D_2}$,
\[  \Pro \left[  f + c_1 R^{1 - \beta/2} \in A \cap B  \right] \le  \Pro \left[ f \in A \right] \Pro \left[ f \in B \right]  + c_1 e^{-c_2 R} . \]
\end{proposition}
\begin{proof}
Set $r = \sqrt{R}$. Similarly to in the proof of Proposition \ref{p.qi}, for any $\ell > 0$, 
\begin{align*}
&  \Pro \left[  f + 2 \ell \in A \cap B  \right] -\Pro \left[ f  \in  A \right] \Pro \left[ f  \in  B \right] \\
& = \Pro \left[  f + 2 \ell \in A \cap B  \right]  -  \Pro \left[  f_r +  \ell \in A \cap B  \right] +  \Pro \left[  f_r +  \ell \in A \right]  \Pro \left[  f_r +  \ell \in B \right] - \Pro \left[  f  \in A  \right]  \Pro \left[  f \in B \right] .
\end{align*}
By using that, for any $a,b,c,d \in [0,1]$, we have $ab-cd \leq (a-c)_+ +(b-d)_+$ and that, for any events $E,F$, we have $(\Pro[E]-\Pro[F])_+ \leq \Pro [E \setminus F]$, we obtain that the above is at most three times the maximum, over all decreasing $C \in \mathcal{F}_{D_1 \cup D_2}$, of 
\begin{equation}
\label{e.psi}
\Pro \left[  f +  \ell \in C   \setminus  f_r \in C   \right] .
 \end{equation}
By Proposition \ref{p.coupling} applied to $t = \sqrt{R}$, there exist $c_1, c_2>0$ such that, outside an event of probability
\[  c_1 e^{-c_2 R} , \]
the fields $f$ and $f_r$ differ, on $D_1 \cup D_2$, by at most 
\[ c_1 \sqrt{R} \, r^{1-\beta} = c_1 R^{1-\beta/2} . \]
Hence with the setting $\ell = c_1 R^{1-\beta/2}$ and by using that $C$ is decreasing,~\eqref{e.psi} is at most $c_1 e^{-c_2 R}$, completing the proof.
\end{proof}

The second auxiliary result bounds the decay of real functions that satisfy a certain functional inequality; we defer its proof until the end of the section.

\begin{lemma}
\label{l.aux1}
Let $(a_R)_{R \ge 0}$ be a positive function such that $a_R \to 0$ and for which there exist $c_1, c_2 , R_0 > 0$ such that, for all $R \ge R_0$,
\[     a_{2R + \sqrt{R}} \le c_1 a_R^2 + e^{-c_2R}  . \]
Then there exist $c_3, c_4 > 0$ and a positive sequence $(m_n)_{n \ge 1}$ such that, for all $n \ge 1$,
\[ 2^n \le m_n \le c_3 2^n \quad \text{and} \quad  a_{m_n} \le e^{-c_4 m_n} . \]
\end{lemma}

\begin{proof}[Proof of Theorem \ref{t.quantfor31}]
Fix $\beta' \in (2, \beta)$, so that $1 - \beta/2 < 1 - \beta'/2 < 0$. Let the level $\ell > 0$ be given, and introduce the increasing sequence of levels $\ell_R  = \ell - R^{1 - \beta'/2}$, which satisfy in particular, for every $c_1 > 0$,
\begin{equation}
\label{e.lr}
\ell_{2R + \sqrt{R}} - \ell_{R} \ge \ell_{2R} - \ell_R = (1 - 2^{1 - \beta'/2})R^{1-\beta'/2} \ge c_1 R^{1 - \beta/2} 
\end{equation}
eventually for large enough $R \ge 1$.

\medskip
Since $ \cross_\ell(2R,R)$ is increasing in $\ell$ and since $\ell_R < \ell$, defining
\[  a_R =  \Pro \left[ f  \notin \cross_{\ell_R}(2R,R) \right]   , \]
  it is sufficient to prove the existence of a $c_1 > 0$ such that, for sufficiently large $R \ge 1$,
\begin{equation} \label{e.for_2R_R}
 a_R  \le  e^{-c_1R} .
\end{equation}
We deduce \eqref{e.for_2R_R} from the following functional inequality for $a_R$, proved immediately below: There exists a $c_1 > 0$ such that, for sufficiently large $R \ge 1$,
\begin{equation}
\label{e.ineq1}
 a_{2R + \sqrt{R}} \le 49 a_R^2 + e^{-c_1R}  .
 \end{equation}
Recalling that Theorem \ref{t.phase_trans} implies that $a_R \to 0$, an application of Lemma \ref{l.aux1} then yields the existence of constants $c_2, c_3 > 0$ and a positive subsequence $(m_n)_{n \ge 1}$ such that, for all $n \ge 1$,
\[ 2^n \le m_n \le c_2 2^n \quad \text{and} \quad  a_{m_n} \le e^{-c_3 m_n} . \]
This implies~\eqref{e.for_2R_R} for $R \in \{ m_n \}_{n \geq 1}$, which can be extended to all $R \ge 0$ by standard gluing arguments.

\medskip 
To prove \eqref{e.ineq1} we introduce two `multiple crossing' events: 
\begin{itemize}
\item  $\multicross_\ell(R)$, which is the union of the following seven events: (i-iv) $\cross_\ell(2R, R)$, and copies of this event translated by $(R, 0)$, $(2R,0)$ and $(3R, 0)$, and (v-vii) $\cross_\ell(R, R)$ translated by $(R, 0)$ and rotated by $\pi/2$, and copies of this event translated by $(R, 0)$ and $(2R, 0)$. This event is depicted at the bottom of Figure~\ref{f.cross1}.
\item  $\multicross'_\ell(R)$, which is the event $\multicross_\ell(R)$ translated by $(0,R+\sqrt{R})$. This event is depicted at the top of Figure~\ref{f.cross1}.
\end{itemize}

\begin{figure}[!h]
\begin{center}
\hspace{3cm} \includegraphics[scale=0.72]{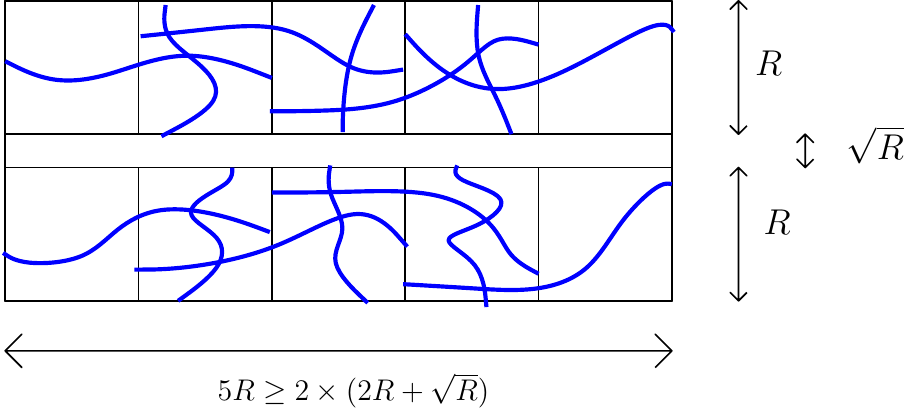}
\end{center}
\caption{The events $\multicross_\ell(R)$ (along the bottom) and $\multicross_\ell'(R)$ (along the top).}\label{f.cross1}
\end{figure}

We also introduce the scales
\[   b_R =  \Pro \left[ f \notin \multicross_{\ell_R}(R) \right] \]
and
\[  b_R' = \Pro \left[ f \notin \multicross_{\ell_{2R+\sqrt{R}}}(R) \cup  \multicross_{\ell_{2R+\sqrt{R}}}'(R) \right] . \]
Now observe the following three facts:
\begin{enumerate}
\item By stationarity and the union bound, $b_R \le 7 a_R$;
\item By stationarity and the `sprinkled' quasi-independence statement in Proposition \ref{p.sqi} (applicable in light of \eqref{e.lr}), there exists $c_1>0$ such that, for sufficiently large $R \ge 1$, 
\[   b_R' \le b_R^2 +  e^{-c_1 R } ; \]
\item For sufficiently large $R \ge 1$,
\[ \multicross_\ell(R) \cup \multicross'_\ell(R) \subseteq  \cross_\ell \left( 2(2R+\sqrt{R}),2R+\sqrt{R}) \right) , \]
(see Figure~\ref{f.cross1}) and so $a_{2R + \sqrt{R}} \le b_R'$. 
\end{enumerate}
Putting these together yields \eqref{e.ineq1}.
 \end{proof}

We deduce the remainder of our results from Theorem \ref{t.quantfor31}, namely Theorem~\ref{t.main1}, Theorem~\ref{t.main1a}, and the second and third statements of Theorem \ref{t.main2} (the first statement of this theorem is given by Theorem~\ref{t.rsw}).

\begin{proof}[Proof of the second and third statements of Theorem \ref{t.main2}]
The result in the case $\ell > 0$ is a consequence of Theorem \ref{t.quantfor31} by standard gluing techniques, and the result in the case $\ell < 0$ then follows since $f$ is equal to $-f$ in law. 
\end{proof}

\begin{proof}[Proof of Theorem~\ref{t.main1}]
The result in the case $\ell \le 0$ follows from the fact that $\Pro \left[ f \in \arm_0(1,R) \right]$ decays to $0$ as $R$ goes to $\infty$ (see the first statement of Theorem~\ref{t.main2}). The result in the case $\ell > 0$ is a consequence of the second statement of Theorem~\ref{t.main2} by standard gluing techniques (see \cite[Lemma 2.8]{rv_bf} for details). Indeed, we have
\[ \sum_{k \in \mathbb{N}} \Pro \left[ \neg \cross_\ell(2^{k+1},2^k) \right] < \infty , \]
and we conclude, by the Borel-Cantelli lemma, that there exists an unbounded connected component (with uniqueness following easily from these arguments as well).
\end{proof}

\begin{proof}[Proof of Theorem \ref{t.main1a}]
The result in the case $\eps = 0$ is immediate since, by Theorem \ref{t.main1}, $\mathcal{E}_0$ does not percolate. In the case $\eps > 0$, we first deduce that, for every quad $Q$, the probability of the crossing $RQ$ by the set $\mathcal{L}_0^\eps$ tends to one at the same rate (given in Theorem \ref{t.main2}) as for the event $\cross_{-\eps}(RQ)$. This is since a crossing of $RQ$ by $\mathcal{L}_0^\eps$ can only not occur if the complementary crossing event for the quad $RQ^\ast$ occurs for \textit{either} the set $\mathcal{E}_{-\eps}$ \textit{or} the set $\mathcal{E}^c_{\eps}$ (which have the same probability). As in the proof of Theorem \ref{t.main1}, standard gluing techniques then give the result (in particular, these gluing arguments do not rely on the FKG inequality in \eqref{e.fkg2}, which is important since $\mathcal{L}_0^\eps$ does not enjoy positive associations).
\end{proof} 

To finish the section, we prove the auxiliary lemma used in the proof of Theorem~\ref{t.quantfor31}:

\begin{proof}[Proof of Lemma \ref{l.aux1}]
Without loss of generality we assume that $c_1 > 1$. Define $a'_R = c_1 (a_R + e^{- (c_2/4) R})$; we prove the result for $a'_R$, which is sufficient since $a_R \le a_R'$.

\medskip
We claim that there exists an $R_1 \geq 2$ such that $a_{R_1}' < 1$ and, for each $R \ge R_1$,
\begin{equation} \label{e.aux1}
   a'_{2R + \sqrt{R}} \le (a'_R)^2 .
   \end{equation}
Such an $R_1$ exists since $a_R \to 0$ and since, for sufficiently large $R \ge 1$,
\begin{align*}
a'_{2R + \sqrt{R}} &= c_1 a_{2R + \sqrt{R}} + c_1e^{-(c_2/4) (2R + \sqrt{R})}  \\
& \le c_1^2 a_R^2 + c_1e^{-c_2R} + c_1e^{-(c_2/2)R} e^{- (c_2/4)\sqrt{R}} \\
& = (c_1 a_R)^2 + c_1( e^{-(c_2/4) R })^2 + c_1e^{-c_2R} -c_1 e^{-(c_2/2)R} (1 - e^{- (c_2/4)\sqrt{R}})  \\
& \le (c_1 a_R)^2 + c_1( e^{-(c_2/4) R })^2 \le (a_R')^2 .
\end{align*}
Now define $m_1= R_1$ and $m_{n+1}=2m_n+ \sqrt{m_n}$. By \eqref{e.aux1} we have, for each $n \ge 1$, $a'_{m_n} \leq (a'_{R_1})^{2^n}$. Moreover, we claim that there exists $c_5>0$ such that $2^n \leq m_n \leq c_5 2^n$. This is easily seen by first establishing the lower bound, and then noting that
\[ \log \Big( \frac{m_{n+1}}{2^{n+1}} \Big) - \log \Big( \frac{m_n}{2^n} \Big) = \log \Big(1 + \frac{1}{2 \sqrt{m_n}} \Big) \le \log \Big(1 + \frac{1}{2^{n/2}} \Big) \le \frac{1}{2^{n/2}} ,\]
which is summable over $n$. Combining these we have the result.
\end{proof}

\bigskip
\bibliographystyle{alpha}
\bibliography{sharpness}

\begin{thebibliography}{DCRT19b}

\bibitem[AB18]{ahlberg2017noise}
D.~Ahlberg and R.~Baldasso.
\newblock Noise sensitivity and {V}oronoi percolation.
\newblock {\em Electron. J. Probab.}, 23, 2018.

\bibitem[Ale96]{alex_96}
K.S. Alexander.
\newblock Boundedness of level lines for two-dimensional random fields.
\newblock {\em Ann. Probab.}, 24(4):1653--1674, 1996.

\bibitem[AT07]{adta_rfg}
R.J. Adler and J.E. Taylor.
\newblock {\em Random fields and geometry}.
\newblock Springer, 2007.

\bibitem[AW09]{azws}
J.~Aza{\"{\i}}s and M.~Wschebor.
\newblock {\em Level sets and extrema of random processes and fields}.
\newblock John Wiley \& Sons, Inc., Hoboken, NJ, 2009.

\bibitem[BDS07]{bds_07}
E.~Bogomolny, R.~Dubertrand, and C.~Schmit.
\newblock {SLE} description of the nodal lines of random wavefunctions.
\newblock {\em J. Phys. A: Math. Theor.}, 40:381--395, 2007.

\bibitem[BEL17]{bel_17}
E.~Di Bernardino, A.~Estrade, and J.R. Le\'{o}n.
\newblock A test of {G}aussianity based on the {E}uler characteristic of
  excursion sets.
\newblock {\em Electron. J. Statist.}, 11(1):843--890, 2017.

\bibitem[BG17]{bg_16}
V.~Beffara and D.~Gayet.
\newblock Percolation of random nodal lines.
\newblock {\em Publ. Math. IHES}, 126:131--176, 2017.

\bibitem[BKS99]{benjamini1999noise}
I.~Benjamini, G.~Kalai, and O.~Schramm.
\newblock Noise sensitivity of {B}oolean functions and applications to
  percolation.
\newblock {\em Publ. Math. IHES}, 90(1):5--43, 1999.

\bibitem[BM18]{bm_17}
D.~Beliaev and S.~Muirhead.
\newblock Discretisation schemes for level sets of planar {G}aussian fields.
\newblock {\em Commun. Math. Phys.}, 359:869--913, 2018.

\bibitem[BMW17]{bmw_17}
D.~Beliaev, S.~Muirhead, and I.~Wigman.
\newblock Russo-{S}eymour-{W}elsh estimates for the {K}ostlan ensemble of
  random polynomials.
\newblock {\em arXiv preprint arXiv:1709.08961}, 2017.

\bibitem[BR06]{bollobas2006percolation}
B.~Bollob{\'a}s and O.~Riordan.
\newblock {\em Percolation}.
\newblock Cambridge University Press, 2006.

\bibitem[BS98]{benjamini1998conformal}
I.~Benjamini and O.~Schramm.
\newblock Conformal invariance of {V}oronoi percolation.
\newblock {\em Commun. Math. Phys.}, 197(1):75--107, 1998.

\bibitem[BS07]{bs_07}
E.~Bogomolny and C.~Schmit.
\newblock Random wavefunctions and percolation.
\newblock {\em J. Phys. A: Math. Theor.}, 40:14033--14043, 2007.

\bibitem[BTA04]{bertho04}
A.~Berlinet and C.~Thomas-Agnan.
\newblock {\em Reproducing Kernel Hilbert Spaces in Probability and
  Statistics}.
\newblock Springer, 2004.

\bibitem[CN07]{camia2007critical}
F.~Camia and C.~M. Newman.
\newblock Critical percolation exploration path and {$SLE_6$}: a proof of
  convergence.
\newblock {\em Probab. Theory Relat. Fields}, 139(3-4):473--519, 2007.

\bibitem[Cuz76]{c_76}
J.~Cuzick.
\newblock A central limit theorem for the number of zeros of a stationary
  {G}aussian process.
\newblock {\em Ann. Probab.}, 4(4):547--556, 1976.

\bibitem[DCRT18]{duminil2018poisson}
H.~Duminil-Copin, A.~Raoufi, and V.~Tassion.
\newblock Subcritical phase of $d$-dimensional {P}oisson-{B}oolean percolation
  and its vacant set.
\newblock {\em arXiv preprint arXiv:1805.00695}, 2018.

\bibitem[DCRT19a]{duminil2017exponential}
H.~Duminil-Copin, A.~Raoufi, and V.~Tassion.
\newblock Exponential decay of connection probabilities for subcritical
  {V}oronoi percolation in {$\R^d$}.
\newblock {\em Probab. Theory Relat. Fields}, 173(1-2):479--490, 2019.

\bibitem[DCRT19b]{duminil2017sharp}
H.~Duminil-Copin, A.~Raoufi, and V.~Tassion.
\newblock Sharp phase transition for the random-cluster and {P}otts models via
  decision trees.
\newblock {\em Ann. Math.}, 189(1):75--99, 2019.

\bibitem[GG06]{graham2006influence}
B.T. Graham and G.R. Grimmett.
\newblock Influence and sharp-threshold theorems for monotonic measures.
\newblock {\em Ann. Probab.}, pages 1726--1745, 2006.

\bibitem[Gri99]{grimmett1999percolation}
G.R. Grimmett.
\newblock {\em Percolation}.
\newblock Springer: Berlin, Germany, 1999.

\bibitem[GS14]{garban2014noise}
C.~Garban and J.~Steif.
\newblock {\em Noise sensitivity of {B}oolean functions and percolation}.
\newblock Cambridge University Press, 2014.

\bibitem[Har60]{harris1960lower}
T.E. Harris.
\newblock A lower bound for the critical probability in a certain percolation
  process.
\newblock {\em Proc. Camb. Phil. Soc.}, 56:13--20, 1960.

\bibitem[Hig02]{h_02}
D.~Higdon.
\newblock Space and space-time modeling using process convolutions.
\newblock In C.W. Anderson, V.~Barnett, P.C. Chatwin, and A.H. El-Shaarawi,
  editors, {\em Quantitative Methods for Current Environmental Issues}. Spring,
  London, 2002.

\bibitem[Jan97]{jan_97}
S.~Janson.
\newblock {\em Gaussian {H}ilbert spaces}.
\newblock Cambridge University Press, 1997.

\bibitem[Kes80]{kesten1980critical}
H.~Kesten.
\newblock The critical probability of bond percolation on the square lattice
  equals {$1/2$}.
\newblock {\em Commun. Math. Phys.}, 74:41--59, 1980.

\bibitem[Kes87]{kesten1987scaling}
H.~Kesten.
\newblock Scaling relations for 2d-percolation.
\newblock {\em Commun. Math. Phys.}, 109(1):109--156, 1987.

\bibitem[KW70]{kimwah70}
G.S. Kimeldorf and G.~Wahba.
\newblock A correspondence between {B}ayesian estimation on stochastic
  processes and smoothing by splines.
\newblock {\em Ann. Math. Statist.}, 41(2):495--502, 1970.

\bibitem[Mal69]{m_69}
T.L. Malevich.
\newblock Asymptotic normality of the number of crossing of level zero by a
  {G}aussian process.
\newblock {\em Theory Probab. Appl.}, 14(2):287--295, 1969.

\bibitem[MS83a]{molchanov1983percolationi}
S.A. Molchanov and A.K. Stepanov.
\newblock Percolation in random fields. {I}.
\newblock {\em Theor. Math. Phys.}, 55(2):478--484, 1983.

\bibitem[MS83b]{molchanov1983percolationii}
S.A. Molchanov and A.K. Stepanov.
\newblock Percolation in random fields. {II}.
\newblock {\em Theor. Math. Phys.}, 55(3):592--599, 1983.

\bibitem[MS86]{molchanov1986percolationiii}
S.A. Molchanov and A.K. Stepanov.
\newblock Percolation in random fields. {III}.
\newblock {\em Theor. Math. Phys.}, 67(2):434--439, 1986.

\bibitem[NS11]{ns_2010}
F.~Nazarov and M.~Sodin.
\newblock Fluctuations in random complex zeroes: asymptotic normality
  revisited.
\newblock {\em Int. Math. Res. Not.}, 2011(24):720--5759, 2011.

\bibitem[NS16]{nazarov2015asymptotic}
F.~Nazarov and M.~Sodin.
\newblock Asymptotic laws for the spatial distribution and the number of
  connected components of zero sets of {G}aussian random functions.
\newblock {\em J. Math. Phys. Anal. Geo.}, 12(3):205--278, 2016.

\bibitem[NSV07]{nsv_2007}
F.~Nazarov, M.~Sodin, and A.~Volberg.
\newblock Transportation to random zeroes by the gradient flow.
\newblock {\em Geom. Funct. Anal.}, 17(3):887--935, 2007.

\bibitem[NSV08]{nsv_2008}
F.~Nazarov, M.~Sodin, and A.~Volberg.
\newblock The {J}ancovici--{L}ebowitz--{M}anificat law for large fluctuations
  of random complex zeroes.
\newblock {\em Commun. Math. Phys.}, 284(3):833--865, 2008.

\bibitem[OSSS05]{o2005every}
R.~O'Donnell, M.~Saks, O.~Schramm, and R.A. Servedio.
\newblock Every decision tree has an influential variable.
\newblock In {\em 46th Annual IEEE Symposium on Foundations of Computer Science
  (FOCS'05)}, pages 31--39, 2005.

\bibitem[Pit82]{pitt1982positively}
L.D. Pitt.
\newblock Positively correlated normal variables are associated.
\newblock {\em Ann. Probab.}, 10(2):496--499, 1982.

\bibitem[Pou99]{poularikas99}
A.D. Poularikas.
\newblock {\em The Handbook of Formulas and Tables for Signal Processing}.
\newblock CRC Press, 1999.

\bibitem[Rod17]{rodriguez20150}
P.F. Rodriguez.
\newblock A 0-1 law for the massive {G}aussian free field.
\newblock {\em Probab. Theory Relat. Fields}, 169:901--930, 2017.

\bibitem[RV19a]{rv_bf}
A.~Rivera and H.~Vanneuville.
\newblock The critical threshold for {B}argmann-{F}ock percolation.
\newblock {\em arXiv preprint arXiv:1711.05012, to appear in Ann. H. Lebesgue},
  2019.

\bibitem[RV19b]{rv_rsw}
A.~Rivera and H.~Vanneuville.
\newblock Quasi-independence for nodal lines.
\newblock {\em arXiv preprint arXiv:1711.05009, to appear in Ann. Inst. H.
  Poincaré Probab. Statist.}, 2019.

\bibitem[RW06]{rw_06}
C.E. Rasmussen and C.K.I. Williams.
\newblock {\em Gaussian processes for machine learning}.
\newblock MIT Press, 2006.

\bibitem[She09]{sheffield99}
S.~Sheffield.
\newblock Exploration trees and conformal loop ensembles.
\newblock {\em Duke Math. J.}, 147(1):79--129, 2009.

\bibitem[Smi07]{smirnov2007towards}
S.~Smirnov.
\newblock Towards conformal invariance of {$2d$} lattice models.
\newblock {\em Proceedings of the ICM}, 2007.

\bibitem[SS09]{ss_99}
O.~Schramm and S.~Sheffield.
\newblock Contour lines of the two-dimensional discrete {G}aussian free field.
\newblock {\em Acta Math.}, 202(1):21--137, 2009.

\bibitem[SS10]{schramm2010quantitative}
O.~Schramm and J.E. Steif.
\newblock Quantitative noise sensitivity and exceptional times for percolation.
\newblock {\em Ann. Math.}, 171(2):619--672, 2010.

\bibitem[SW01]{Smirnov2001critical}
S.~Smirnov and W.~Werner.
\newblock Critical exponents for two-dimensional percolation.
\newblock {\em Math. Res. Lett.}, 8(5-6):729--744, 2001.

\bibitem[Tas16]{tassion2014crossing}
V.~Tassion.
\newblock Crossing probabilities for {V}oronoi percolation.
\newblock {\em Ann. Probab.}, 44(5):3385--3398, 2016.

\bibitem[Wei82]{w_82}
A.~Weinrib.
\newblock Percolation threshold of a two-dimensional continuum system.
\newblock {\em Phys. Rev. B}, 26(3):1352--1361, 1982.

\bibitem[Wei84]{weinrib1984long}
A.~Weinrib.
\newblock Long-range correlated percolation.
\newblock {\em Phys. Rev. B}, 29(1):387, 1984.

\bibitem[Wen05]{wendland05}
H.~Wendland.
\newblock {\em Scattered Data Approximation}.
\newblock Cambridge Monographs on Applied and Computational Mathematics.
  Cambridge University Press, 2005.

\end{thebibliography}

\end{document}